\documentclass[reqno,twoside,11pt]{amsart}
\usepackage{cite}
\usepackage{amsmath}
\usepackage{amsfonts}
\usepackage{amssymb}
\usepackage{epsfig}
\usepackage{verbatim}
\usepackage{color}

\usepackage{mathpazo}
\usepackage{mathabx}
\usepackage{mathrsfs}

\DeclareFontFamily{OT1}{cmss}{} \DeclareFontShape{OT1}{cmss}{m}{n} {<5> <6> <7> <8> <9> <10> <11> <12> <13> <14.4> cmss10}{}
\DeclareMathAlphabet{\cmss}{OT1}{cmss}{m}{n}

\DeclareFontFamily{OT1}{fraktura}{}
\DeclareFontShape{OT1}{fraktura}{m}{n} {<5> <6> <7> <8> <9> <10> <11> <12> <13> <14.4> [1.1] eufm10}{}
\DeclareMathAlphabet{\fraktura}{OT1}{fraktura}{m}{n}

\newtheoremstyle{thm}{1.5ex}{1.5ex}{\itshape\rmfamily}{} {\bfseries\rmfamily}{}{2ex}{}

\newtheoremstyle{def}{1.5ex}{1.5ex}{\rmfamily\sl}{} {\bfseries\rmfamily}{}{2ex}{}

\newtheoremstyle{rem}{1.3ex}{1.3ex}{\rmfamily}{} {\bfseries\rmfamily}{}{2ex}{}

\newtheoremstyle{ass}{1.5ex}{1.5ex}{\rmfamily\sl}{} {\bfseries\rmfamily}{}{2ex}{}

\newenvironment{proofsect}[1] {\vskip0.1cm\noindent{\rmfamily\itshape#1.}}{\qed\vspace{0.15cm}}

\theoremstyle{thm}
\newtheorem{theorem}{Theorem}[section]
\newtheorem{lemma}[theorem]{Lemma}
\newtheorem{proposition}[theorem]{Proposition}

\newtheorem*{Main Theorem}{Main Theorem.}
\newtheorem{corollary}[theorem]{Corollary}
\newtheorem{conjecture}[theorem]{Conjecture}

\theoremstyle{rem}
\newtheorem{remark}[theorem]{{Remark}}

\numberwithin{equation}{section}

\renewcommand{\section}{\secdef\sct\sect}
\newcommand{\sct}[2][default]{\refstepcounter{section}
\addcontentsline{toc}{section}
{{\tocsection {}{\thesection}{\!\!\!\!#1\dotfill}}{}}
\vspace{0.7cm}
\centerline{ 
\scshape\arabic{section}.\ #1} \nopagebreak \vspace{0.2cm}}
\newcommand{\sect}[1]{
\vspace{0.4cm} \centerline{\large\scshape\rmfamily #1}
\vspace{0.2cm}}

\renewcommand{\subsection}{\secdef\subsct\sbsect}
\newcommand{\subsct}[2][default]{\refstepcounter{subsection}
\addcontentsline{toc}{subsection}
{{\tocsection{\!\!}{\hspace{1.2em}\thesubsection}{\!\!\!\!#1\dotfill}}{}}
\nopagebreak\vspace{0.45\baselineskip} {\flushleft\bf
\thesection.\arabic{subsection}~\bf #1.~}
\\*[3mm]\noindent
\nopagebreak}
\newcommand{\sbsect}[1]{
\vspace{0.1cm}\noindent
\textbf{#1.~}\vspace{0.1cm}}

\renewcommand{\subsubsection}{%
\secdef \subsubsect\sbsbsect}
\newcommand{\subsubsect}[2][default]{%
\refstepcounter{subsubsection} 
\addcontentsline{toc}{subsubsection}{{\tocsection{\!\!}
{\hspace{3.05em}\thesubsubsection}{\!\!\!\!#1\dotfill}}{}}
\nopagebreak
\vspace{0.15\baselineskip} \nopagebreak {\flushleft\rmfamily
\itshape\arabic{section}.\arabic{subsection}.\arabic{subsubsection}
\ \rmfamily #1\/.}\ }
\newcommand{\sbsbsect}[1]{\vspace{0.1cm}\noindent
\rmfamily \itshape
\arabic{section}.\arabic{subsection}.\arabic{subsubsection} \
\sffamily #1\/.\ }

\renewcommand{\caption}[1]{%
\vglue0.5cm
\refstepcounter{figure}
\begin{center}
\begin{minipage}[c]{0.8\textwidth}\small {\sc Fig.~\thefigure\ }#1\end{minipage}
\end{center}
}

\newcommand{\textd}{\text{\rm d}\mkern0.5mu}

\newcommand{\texte}{\text{\rm  e}\mkern0.7mu}

\newcommand{\Var}{\text{\rm Var}}

\renewcommand{\AA}{\mathcal A}
\newcommand{\BB}{\mathcal B}
\newcommand{\CC}{\mathcal C}

\newcommand{\EE}{\mathcal E}
\newcommand{\FF}{\mathcal F}
\newcommand{\GG}{\mathcal G}
\newcommand{\HH}{\mathcal H}

\newcommand{\LL}{\mathcal L}

\newcommand{\NN}{\mathcal N}

\newcommand{\ZZ}{\mathcal Z}

\newcommand{\E}{\mathbb E}
\newcommand{\bbE}{\E}

\newcommand{\BbbL}{\mathbb L}

\newcommand{\N}{\mathbb N}

\newcommand{\BbbP}{\mathbb P}
\newcommand{\bbP}{\BbbP}

\newcommand{\R}{\mathbb R}
\newcommand{\bbR}{\R}

\newcommand{\T}{\mathbb T}

\newcommand{\V}{\mathbb V}

\newcommand{\Z}{\mathbb Z}

\newcommand{\scrF}{\mathscr{F}}

\newcommand{\twoeqref}[2]{(\ref{#1}--\ref{#2})}

\newcommand{\cc}{{\text{\rm c}}}

\def\myffrac#1#2 in #3{\raise 2.6pt\hbox{$#3 #1$}\mkern-1.5mu\raise 0.8pt\hbox{$#3/$}\mkern-1.1mu\lower 1.5pt\hbox{$#3 #2$}}

\newcommand{\wh}{\widehat}
\newcommand{\wt}{\widetilde}

\newcommand{\laweq}{\,\overset{\text{\rm law}}=\,}

\newcommand{\Lawarrow}{{\,\overset{\text{\rm law}}\longrightarrow\,}}

\newcommand{\rmd}{\textd}
\newcommand{\rme}{\texte}
\newcommand{\rmc}{\cc}

\newcommand{\slb}{\sqrt{\log b}}
\newcommand{\PPP}{\text{\rm PPP}}

\begin{document}


\title[ Favorite point of random walk\hfill]{
A limit law for the most favorite point\\of simple random walk on a regular tree}
\author[\hfill M.~Biskup and O. Louidor]
{Marek~Biskup$^{1}$\, and\, Oren Louidor$^2$}
\thanks{\hglue-4.5mm\fontsize{9.6}{9.6}\selectfont\copyright\,\textrm{2021}\ \ \textrm{M.~Biskup, O. Louidor. Reproduction, by any means, of the entire
article for non-commercial purposes is permitted without charge.\vspace{2mm}}}
\maketitle

\vspace{-5mm}
\centerline{\textit{$^1$
Department of Mathematics, UCLA, Los Angeles, California, USA}}
\centerline{\textit{$^2$
Faculty of Industrial Engineering and Management, Technion, Haifa, Israel}}
\smallskip


\vskip0.5cm
\begin{quote}
\footnotesize \textbf{Abstract:}
We consider a continuous-time random walk on a regular tree of finite depth and study its favorite points among the leaf vertices. We prove that, for the walk started from a leaf vertex and stopped upon hitting the root, as the depth of the tree tends to infinity the maximal time spent at any leaf converges, under suitable scaling and centering, to a randomly-shifted Gumbel law. The random shift is characterized using a derivative-martingale like object associated with square-root local-time process on the tree.
\end{quote}


\section{Introduction and results}
\vglue-3mm\subsection{Background}
\noindent
Extremal properties of random walks have been a source of continuing attention of probabilists for several decades. One such property is the time spent by the walk at its most favorite points which are those visited most frequently over a given time period. The study of favorite points was initiated by Erd\H os and Taylor~\cite{ET60} who analyzed the leading-order $n$-dependence of the time~$T^\star_n$ that the simple symmetric random walk on~$\Z^d$ spends at its most visited point by time~$n$. The approach of~\cite{ET60}, which relied on treating the number of visits to a point by portions of the walk as sums of independent geometric random variables, showed that~$T^\star_n/\log n$ tends to a computable limit as~$n\to\infty$ in the transient dimensions $d\ge3$, but gave only asymptotic bounds $\frac1{4\pi}(\log n)^2\lesssim T_n^\star\lesssim\frac1\pi(\log n)^2$ in the recurrent dimension $d=2$.

Erd\H os and Taylor conjectured their upper bound to be sharp but this was settled only four decades later by Dembo, Peres, Rosen and Zeitouni~\cite{DPRZ01} using excursion decomposition along a logarithmic sequence of scales and second-moment computations for the occupation measure. (Most of~\cite{DPRZ01} actually deals with two-dimensional Brownian motion.) In addition to proving that $T_n^\star/(\log n)^2$ tends to $1/\pi$ in probability, \cite{DPRZ01} offered insight into the structure of the $\lambda$-favorite points which, for $\lambda\in(0,1)$, are those visited by the walk at least~$\lambda \frac1\pi(\log n)^2$ times. 
For instance,  it showed that  there are $n^{1-\lambda+o(1)}$ such points in a path of time length~$n$ and, as seen in the proofs, the spatial distribution of these points has an intriguing fractal structure.

The understanding of $\lambda$-favorite points has further advanced thanks to Jego~\cite{J19} who recently established a weak limit for scaled empirical measures associated with $\lambda$-favo\-rite points (for $0<\lambda<1$) of the simple random walk on~$\Z^2$ stopped upon exit from a scaled-up lattice version of a continuum planar domain. The limit measure --- dubbed Brownian Multiplicative Chaos in~\cite{J19} --- is similar to, albeit distinct from, the scaling limit of the thick points of the Gaussian Free Field (GFF) derived by the present authors in~\cite{BL4}. The connection to GFF thick points is much stronger once the random walk is run (via a suitable return mechanism) for times comparable with the cover time. This is the subject of the work of Abe, Lee and the first author~\cite{AB,ABL}. 

Some progress has occurred also for the the time spent at the most favorite points. Indeed, Abe~\cite[Corollary 1.3]{A15} proved a result analogous to~\cite{DPRZ01} for the walk on a two-dimensional torus run for times comparable to the cover time. Jego~\cite[Theorem~1.1.1]{J18} extended the conclusions of \cite{DPRZ01} to a large class of random walks. In~\cite{J20}, Jego in turn constructed a candidate for the measure that should govern (similarly to the results on the extrema of GFF by the present authors~\cite{BL1,BL2,BL3}) the distribution of the extremal process associated with the most favorite points  of simple random walk on~$\Z^2$.  Still, the control of the time spent at the most favorite points is presently limited to leading-order asymptotic. Indeed, even the second-order term in the centering sequence remains conjectural, let alone a proof of tightness and/or a distributional scaling limit. 
 
Our goal here is to take up the problem of the time spent at the most favorite points of the random walk on a simpler graph; namely, a regular tree. This walk shares some of the basic features of the random walk on~$\Z^2$ but is easier to study thanks for the Markovian structure of (suitably parametrized) local time on the tree.

\subsection{Most favorite leaf on regular tree}
Let~$\T_n$ be the regular tree of depth~$n$ with forward degree~$b\ge2$ and the root denoted by~$\varrho$. Write~$\BbbL_n$ for the set of its leaves; namely, the set of~$b^n$ vertices at graph-theoretical distance~$n$ from~$\varrho$. It is well known that the projection of the random walk on~$\T_n$ to its leaf vertices carries a lot of similarities to the walk on~$\Z^2$. Indeed,~$\BbbL_n$ can, for $b=4$, be identified with the square box~$\Lambda_n$ in~$\Z^2$ of side-length~$2^n$. The Green function on~$\BbbL_n$, defined as the expected number of visits to one leaf vertex for the walk started at another and killed upon hitting~$\varrho$,  has a similar logarithmic structure as that on $\Lambda_n$, albeit relative to the ultrametric (tree) distance instead of the Euclidean norm. The first exit from~$\Lambda_n$ by the random walk on~$\Z^2$ corresponds to the walk on~$\T_n$ hitting~$\varrho$.

For simplicity of exposition and technical advantage later, we will run the random walk in continuous time. Let $\{X_t\}_{t\ge0}$ be the continuous-time random walk on~$\T_n$ with a unit transition rate across each edge and, for $z\in\T_n$, let~$P^z$ denote the law of this walk subject to $P^z(X_0=z)=1$. For each~$x\in\T_n$ denote by
\begin{equation}
\label{E:1.1q}
\ell_t(x):=\int_0^t1_{\{X_s=x\}}\textd s
\end{equation}
the total time~$X$ has spent at~$x$ by time~$t$ and let
\begin{equation}
\label{E:1.2q}
\tau_\varrho:=\inf\{t\ge0\colon X_t=\varrho\}
\end{equation}
be the first hitting time of the root. Our first result is then:

\begin{theorem}
\label{thm-1}
For any~$x_n\in\BbbL_n$ and all~$u\in\R$,
\begin{equation}
\label{E:1.3}
P^{x_n}\biggl(\max_{x\in\BbbL_n}\sqrt{\ell_{\tau_\varrho}(x)}\le \slb\, n-\frac1{\slb}\log n +u\biggr)
\,\,\underset{n\to\infty}\longrightarrow\,\,\E\bigl(\texte^{-\ZZ \texte^{-2u\sqrt{\log b}}}\bigr),
\end{equation}
where~$\ZZ$ is an a.s.-positive and finite random variable. In particular,
\begin{equation}
\label{E:1.4}
\frac1n\biggl(\max_{x\in\BbbL_n}\ell_{\tau_\varrho}(x) -\bigl(n^2\log b -2n\log n\bigr)\biggr) 
\,\,\underset{n\to\infty}\Lawarrow\,\,\log\ZZ+G,
\end{equation}
where~$G$ is a normalized Gumbel random variable independent of $\ZZ$.
\end{theorem}

By the symmetries of~$\T_n$, the probability on the left of \eqref{E:1.3} is the same for all starting points $x_n\in\BbbL_n$. In order to relate the conclusions to the aforementioned results for the simple random walk on~$\Z^2$, note that $\frac1n\log\tau_\varrho\to\log b$ in $P^{x_n}$-probability. The leading order growth rate of the maximal local time is thus proportional to $(\log\tau_\varrho)^2$, in accord with Erd\H os and Taylor's calculations on~$\Z^2$.

\subsection{Random walk started from the root}
The recursive structure of the tree naturally guides us to consider the corresponding problem for the random walk started from the root. This turns out to be particularly convenient for the local time parametrized by the time spent at the root. To convert to this parametrization we need
\begin{equation}
\wt\tau_\varrho(t):=\inf\bigl\{s\ge0\colon\ell_s(\varrho)> t\bigr\}.
\end{equation}
Then
\begin{equation}
\label{E:1.8u}
L_t(x):=\ell_{\wt\tau_\varrho(t)}(x)
\end{equation}
is exactly the time spent by the walk at~$x$ when the time spent at the root reaches~$t$. The definition gives~$L_t(\varrho)=t$ and, as is well known, $E^\varrho(L_t(x))=t$ for all~$x\in\T_n$. 

A key technical advantage of the parametrization \eqref{E:1.8u} is that, for the walk started at the root, the family of random variables $\{L_t(x)\}_{x\in\T_n}$ has a Markov property  under restrictions to subtrees; see Lemma~\ref{lemma-M} for details. Consequently, the laws of $\{L_t(x)\}_{x\in\T_n}$ for different~$n$'s are consistent under~$P^\varrho$ and are thus restrictions from a unique law on the infinite $b$-ary tree.

In order to describe our main result concerning the local times $\{L_t(x)\}_{x\in\BbbL_n}$, for an integer~$n\ge1$, a real number~$t>0$ and the random walk sampled from~$P^\varrho$, set
\begin{equation}
\label{E:1.9a}
Z_n(t) := b^{-2n} \sum_{x \in \BbbL_n} \Bigl(n\slb - 
\sqrt{L_t(x)}\,\Bigr)^+ L_t(x)^{1/4}\,
	\rme^{2\slb\,\sqrt{L_t(x)}},
\end{equation}
where ``$+$'' denotes the positive part. We then have:

\begin{theorem}
\label{thm-2}
For all $t > 0$, there exists an a.s.-finite non-negative random variable $Z(t)$ with $\BbbP(Z(t)>0)\in(0,1)$ such that
\begin{equation}
\label{E:2.2}
Z_n(t) \overset{\rm{law}}{\underset{n \to \infty}\longrightarrow} Z(t) \,.
\end{equation}
Moreover, for all~$t>0$,
\begin{equation}
\label{E:2.2a}
\lim_{n\to\infty}P^\varrho\Bigl(\,\max_{x\in\BbbL_n}L_t(x)>0\Bigr) = \BbbP\bigl(Z(t)>0\bigr)
\end{equation}
and there exists a constant $C_\star \in (0,\infty)$ such that for all $u \in \bbR$,
\begin{equation}
\label{E:2.3}
P^\varrho \biggl(\max_{x\in\BbbL_n} \sqrt{L_t(x)} \le n\slb-\frac1{\slb}\log n + u\biggr)\,\,
\underset{n \to \infty} \longrightarrow\,\,
\bbE \bigl(\,\rme^{-C_\star Z(t) \rme^{-2u\slb}}\,\bigr).
\end{equation}
Here~$\BbbP$ denotes the law of~$Z(t)$ and~$\E$ is the expectation with respect to~$\BbbP$.
\end{theorem}

The limit quantity in~\eqref{E:2.3} can be viewed two different ways. First, it is the Laplace transform of the law of~$Z(t)$ evaluated at $C_\star\rme^{-2u\slb}$ which, we note, varies through the positive reals as~$u$ varies through~$\R$. Second, it is the CDF of a defective random variable taking values in $\R \cup \{-\infty\}$ which, conditional on finiteness, is Gumbel with rate $2\slb$ shifted by $(2\slb)^{-1} \log Z(t)$. The defect comes from an atom at $-\infty$ of mass $\bbP(Z(t) = 0)$. By \eqref{E:2.2a}, conditioning  on $Z(t)>0$ amounts to conditioning on the random walk to hit the leaf vertices before accumulating time~$t$ at the root. 

\subsection{Connection to Branching Random Walk}
Randomly shifted Gumbel limit laws for centered maxima have been encountered in a number of contexts. These include Branching Brownian Motion (Bramson~\cite{Bramson1,Bramson2}) and critical Branching Random Walks (A\"idekon~\cite{Aidekon}, Bramson, Ding and Zeitouni~\cite{BDZ-BRW}) as well as the two-dimensional discrete GFF (Bramson, Ding and Zeitouni~\cite{BDingZ}, Biskup and Louidor~\cite{BL1,BL2,BL3}) and other logarithmically correlated processes (e.g., Madau\-le~\cite{Madaule2}, Ding, Roy and Zeitouni~\cite{DRZ}, Arguin and Oumet~\cite{Arguin-Oumet}, Schweiger~\cite{Schweiger}, Fels and Hartung~\cite{Fels-Hartung}) including the local time for  our  simple random walk on~$\T_n$ run for times comparable with the cover time (Abe~\cite{A18}).

The case particularly relevant for our problem is the Branching Random Walk (BRW) with step distribution~$\NN(0,1/2)$, also known as the GFF scaled by~$1/\sqrt2$. The latter is a Gaussian process $\{h_x\}_{x\in\T_n}$ defined by sampling an independent copy of~$\NN(0,1/2)$ for each edge of the tree and, for $x\in\T_n$, letting~$h_x$ denote the sum of these variables along the unique path from the root to~$x$. For this process, A\"idekon's result reads:
\begin{equation}
\label{E:1.12i}
P\biggl(\max_{x\in\BbbL_n}h_x\le n\slb-\frac34\frac1{\slb}\log n +u\biggr)
\longrightarrow\,\,
\bbE \bigl(\,\rme^{-C_\star' W \rme^{-2u\slb}}\,\bigr),
\end{equation}
where~$W$ is the weak limit of the sequence
\begin{equation}
\label{E:1.15}
W_n:=b^{-2n}\sum_{x\in\BbbL_n}\bigl(n\slb - h_x\bigr)\texte^{2\slb\,h_x}
\end{equation}
and $C_\star'$ is a positive constant that can be characterized; see Abe~\cite[Remark~1.4]{A18}.

The convergence $W_n\to W$ relies on the fact that~$\{W_n\}_{n\ge1}$, if interpreted on the infinite $b$-ary tree, is a martingale with respect to a natural filtration. (In this framework the limit exists in a.s.~sense.) This is not special to Gaussian step distributions; indeed, general conditions ensuring existence and non-triviality of the limit for general BRW exist (see, e.g., Biggins and Kyprianou~\cite[Theorems~5.1-5.2]{Biggins-Kyprianou}). The limit object~$W$ earns the name \emph{derivative martingale} through the fact that~$W_n$ can be obtained by differentiating $\beta\mapsto\sum_{x\in\BbbL_n}\exp\{\beta h_x-\tfrac12\beta^2 n\}$ at $\beta:=2\slb$.

While $Z_n(t)$ is similar to~$W_n$ in form, the sequence $\{Z_n(t)\}_{n\ge1}$ is not a martingale (under the law of the local time on the infinite $b$-ary tree) due to the more complicated dependency structure of~$L_t$. The weak convergence \eqref{E:2.2} thus has to be established as part of the proof of Theorem~\ref{thm-2}. Notwithstanding, the local time $L_t$ behaves quite similarly to the BRW when~$t$ is large. Indeed, for $n\ge1$ fixed, the Multivariate CLT along with the Kac Moment Formula (Kac~\cite{Kac1,Kac2}) yield
\begin{equation}
\label{E:1.13r}
\bigl\{\sqrt{L_t(x)}-\sqrt t\,\bigr\}_{x\in\T_n}\,\,\underset{t\to\infty}\Lawarrow\,\,\{h_x\}_{x\in\T_n}.
\end{equation}
In light of this we would expect that \eqref{E:2.3} somehow morphs into \eqref{E:1.12i} in the limit as~$t\to\infty$. This is true, albeit not without correction terms:

\begin{theorem}
\label{thm-1.3}
For~$Z(t)$ and~$W$ as defined above,
\begin{equation}
\label{E:1.18a}
t^{-1/4}\,\texte^{-2\slb\,\sqrt t}\,Z(t)\,\underset{t\to\infty}\Lawarrow\, W.
\end{equation}
Moreover, the constants in \eqref{E:2.3} and \eqref{E:1.12i} obey $C_\star'=C_\star$.
\end{theorem}

Both Theorems~\ref{thm-2} and~\ref{thm-1.3} will be extracted from Theorem~\ref{thm-5} which asserts convergence of the kind \eqref{E:2.3} \emph{uniformly} in all~$t=o(n^2)$. In this case the maximum of~$\sqrt{L_t(x)}$ must be centered by a suitable $t$-dependent variant of the centering sequence in \eqref{E:2.3} which, roughly speaking, allows for a smooth cross-over between the second-order terms in \eqref{E:2.3} and \eqref{E:1.12i} when~$t$ increases as a power of~$n$. For a fixed~$t$, the  change in the centering  sequence tends (as~$n\to\infty$) to a $t$-dependent constant which, if transferred  through  the underlying variable~$u$ to the right-hand side, gives rise to the prefactors of~$Z(t)$ in \eqref{E:1.18a}.

To identify the constants in \eqref{E:1.12i} and \eqref{E:2.3} we actually rely on Corollary~1.3 of Abe~\cite{A18} who proved the convergence in Theorem~\ref{thm-2} with  the aforementioned $t$-dependent centering sequence  and~$Z(t)$ replaced by~$W$ in the regime when~$t\ge c_1n\log n$ for some~$c_1>0$. An interesting additional phenomenon in this regime is that, if $t/n^2$ remains bounded away from zero, the ``constant''~$C_\star$ picks up dependence on the asymptotic value of~$t/n^2$.

\subsection{Random shift}
For the walk started from the root and parametrized by the time spent there, Theorem~\ref{thm-2} characterizes the law of the scaled maximum as Gumbel with a random shift proportional to~$\log Z(t)$. In light of this it seems reasonable to ask: What is the distribution of the variable~$\ZZ$ in Theorem~\ref{thm-1}? Can it be characterized by a similar limit expression as~$Z(t)$? How is the law of~$\ZZ$ related to the laws of~$Z(t)$?

To answer these questions, we need additional observations about the~$Z(t)$'s. Recall that a random variable is \emph{Compound Poisson-Exponential} with parameter~$t$ if it has the law of $\sum_{i=1}^{N(t)}U_i$ for $N(t)$ Poisson with parameter~$t$ and $\{U_i\}_{i\ge1}$ independent i.i.d.\ Exponentials with parameter~$1$.  Next note  the following consequence of Theorem~\ref{thm-2}:

\begin{corollary}
\label{thm-2b}
For each Borel $E\subseteq\R^2$, the map $t\mapsto \BbbP((t,Z(t))\in E)$ is Borel measurable. In particular, each positive random variable~$T$ can be coupled with a unique random variable~$Z(T)$ such that
\begin{equation}
\label{E:1.13}
P\bigl((T,Z(T))\in \cdot\bigr) = \int \BbbP\bigl((t,Z(t))\in \cdot\bigr)P(T\in\textd t)
\end{equation}
and, assuming~$T$ to be independent of~$Z_n$,
\begin{equation}
\label{E:1.14w}
Z_n(T)\,\underset{n\to\infty}\Lawarrow\, Z(T).
\end{equation}
Moreover, for each $t>0$, the cascade relation holds
\begin{equation}
\label{E:1.15w}
Z(t) \laweq \sum_{i=1}^b b^{-2} Z(T_i),
\end{equation}
where $Z(T_1),\dots,Z(T_b)$ are i.i.d.\ copies of random variable~$Z(T)$ for $T$ Compound Poisson-Exponential with parameter~$t$.
\end{corollary}

The slightly cumbersome description of the joint law of~$(T,Z(T))$ stems from our present inability to extend the convergence \eqref{E:2.2} to that of a full process $t\mapsto Z_n(t)$. Indeed, if we had the limit process $t\mapsto Z(t)$ at our disposal, $Z(T)$ could be defined directly by evaluating $t\mapsto Z(t)$ at~$t:=T$, for~$T$ independent of~$Z(\cdot)$.

We now characterize the law of~$\ZZ$ three possible ways:

\begin{theorem}
\label{thm-3}
For the constant~$C_\star$ and random variable $Z(t)$ as in Theorem~\ref{thm-2}, the random variable~$\ZZ$ from Theorem~\ref{thm-1} has the law of
\begin{enumerate}
\item[(1)] the weak limit of $C_\star Z(t)$ conditional on $Z(t)>0$ as $t\downarrow0$,
\item[(2)] $C_\star\, b^{-2} Z(U)$ conditioned on $Z(U)>0$ for~$U=$ Exponential with parameter~$1$,
\item[(3)] the weak limit (as $n\to\infty$) of the sequence
\begin{equation}
\label{E:1.9c}
C_\star\, b^{-2n} \sum_{x \in \BbbL_n} \Bigl(n\slb - 
\sqrt{\ell_{\tau_\varrho}(x)}\,\Bigr)^+ \ell_{\tau_\varrho}(x)^{1/4}\,
	\rme^{2\slb\,\sqrt{\ell_{\tau_\varrho}(x)}}
\end{equation}
under $P^{x_n}$, for any~$x_n\in\BbbL_n$,
\end{enumerate}
where all the stated weak limits exist.
\end{theorem}

Note that while~(3) brings $\ZZ$ very close to~$Z(t)$ and thus also  to  the derivative martingale~$W$, neither~(1) nor~(3) make it apparent that~$\ZZ>0$ a.s.

\subsection{Remarks, questions and conjectures}
We finish with some additional remarks on our results and directions of future study.
Having described the law of  the  maximal local time, the next natural step is to investigate the  extremal process associated with near-maximal points. For a fixed~$n$, this is captured by the Radon measure on $[0,1]\times\R$ defined by
\begin{equation}
\label{E:1.20}
\sum_{x\in\BbbL_n}\delta_{\theta_n(x)}\otimes\delta_{\sqrt{L_t(x)}-n\slb-(\log b)^{-1/2}\log n}\,,
\end{equation}
where, for~$x\in\BbbL_n$ represented by a sequence $(\sigma_1,\dots,\sigma_n)\in\{1,\dots,b\}^n$ marking which branch the unique path from~$\varrho$ to~$x$ takes at each step,~$\theta_n(x):=\sum_{i=1}^n\sigma_i b^{-n-i-1}$ maps~$x$ canonically to the unit interval~$[0,1]$. 

In analogy with our earlier work on  two-dimensional  GFF~\cite{BL1,BL2,BL3}, the work of Arguin, Bovier and Kistler~\cite{ABK1,ABK2,ABK3} and A\"id\'ekon, Berestycki, Brunet and~Shi~\cite{ABBS} on Branching Brownian Motion, the work of Madaule~\cite{Madaule2} on Branching Random Walk and Abe's work~\cite{A18} on the local time process on the tree in the regime when~$t$ increases to infinity with~$n$, we expect \eqref{E:1.20} to converge weakly to a \emph{clustered Cox process} of the form
\begin{equation}
\label{E:1.20w}
\sum_{i,j\ge1}\delta_{x_i}\otimes\delta_{u_i-\phi_{i,j}}.
\end{equation}
Here $\{(x_i,u_i,\{\phi_{i,j}\}_{j\ge1}\}_{i\ge1}$ lists the sample points of the Poisson point process
\begin{equation}
\label{E:1.21}
\PPP\Bigl(\,\mu_t(\textd x)\otimes\texte^{-2u\slb}\textd u\otimes\nu(\textd\phi)\Bigr)
\end{equation}
on $[0,1]\times\R\times[0,\infty)^{\N}$ where~$\mu_t$ is a random Borel measure on~$[0,1]$ such that
\begin{equation}
\label{E:1.22a}
\mu_t\bigl([0,1]\bigr)\,\laweq\, 2\slb \,C_\star Z(t)
\end{equation}
while $\nu$ is a deterministic probability law on locally finite (but a.s.-infinite) point process on~$[0,\infty)$. The Poisson point process in \eqref{E:1.21} is sampled conditionally on~$\mu_t$.

Note that all the~$t$ dependence of \eqref{E:1.21} rests in the measure~$\mu_t$. Moreover, the cluster process law~$\nu$ should be the same as for the GFF on the tree. This has already been verified in the regime when~$t$ increases with~$n$ faster than~$c n\log n$ (Abe and Biskup, private communication). In accord with Theorem~\ref{thm-1.3} we also expect that $\mu_t$ scaled by the prefactors in \eqref{E:1.18a} converges weakly to the critical Gaussian Multiplicative Chaos measure associated with the GFF on~$b$-ary tree. A similar conclusion as \twoeqref{E:1.20w}{E:1.22a} should hold also for the setting of Theorem~\ref{thm-1}, except that~$\mu_t$ gets replaced by a measure derived from~$\mu_t$ by the same procedure as~$\ZZ$ is derived from~$Z(t)$. 

The statements \twoeqref{E:1.20w}{E:1.22a} are consistent with Jego's proposal (see~\cite[ Conjecture~1]{J20}) for the weak limit of centered maximal local time of the simple random walk run until the first exit from a square domain in~$\Z^2$. For the construction of the purported limit measure (which is the main conclusion of~\cite{J20}) Jego works directly in the continuum using Brownian motion instead of random walk while generalizing ideas from the study of Brownian thick points where the corresponding {Brownian Multiplicative Chaos} measures were  first  constructed by Bass, Burdzy and Khoshnevisan~\cite{BBK} in the (so called) second-moment regime, and then by Jego~\cite{J18b} and, independently, A\"idekon, Hu and Shi~\cite{AHS} throughout the subcritical regime.

One of Jego's constructions of the critical Brownian Multiplicative Chaos relies on Seneta-Heyde norming which amounts to replacing the polynomial terms in front of the exponentials in \eqref{E:1.9a} by deterministic~$n$-dependent terms. An important point noted in~\cite{J20} is that the Seneta-Heyde norming requires a different multiplier for the local time object than for the GFF derivative martingale. Indeed, for GFF the norming compensates for the term in the parenthesis in \eqref{E:1.15} which is known to be typically of order~$\sqrt n$. We expect the corresponding term in~$Z_n(t)$ to be of the same order but then we need another factor~$\sqrt n$ to account for the term $L_t(x)^{1/4}$ (as this should be dominated by the leading order of the maximum). 
Thus we expect:

\begin{conjecture}
There is~$c>0$ such that for all~$t>0$, under $P^\varrho$,
\begin{equation}
\label{E:1.23}
n b^{-2n} \sum_{x \in \BbbL_n} 
	\rme^{2\slb\,\sqrt{L_t(x)}}\,\underset{t\to\infty}\Lawarrow\, c Z(t)
\end{equation}
\end{conjecture}

We believe that the convergence \eqref{E:1.23} occurs simultaneously for all~$t\ge0$, under the natural coupling of $L_t$ on the infinite tree (under~$P^\varrho$) for all times. The resulting process $t\mapsto Z(t)$ is then naturally monotone in~$t$. The conclusion should extend to the measure~$\mu_t$ in \eqref{E:1.21} by representing it, modulo overall normalization, as the limit of the measures
\begin{equation}
\label{E:1.24w}
n b^{-2n} \sum_{x \in \BbbL_n} 
	\rme^{2\slb\,\sqrt{L_t(x)}}\delta_{\theta_n(x)},
\end{equation}
where~$\theta_n(x)$ was defined after \eqref{E:1.20}. Another version of \eqref{E:1.24w} could include information about the local time at the points that carry the support of the limiting measure. For the two-dimensional GFF, this has been done by the authors jointly with S.~Gufler~\cite{BGL}. Our results there in fact show that \emph{any} polynomial prefactor that reproduces the same deterministic scaling will lead to a multiple of the same measure.

We remain puzzled by the fact that the bulk of our derivations make no significant use of the strong connection between the local time and the GFF known as the Second Ray-Knight Theorem and/or Dynkin Isomorphism (Eisenbaum, Kaspi, Markus, Rosen and~Shi~\cite{EKMRS}, Dynkin~\cite{Dynkin}). This connection turned out to be extremely useful in the study of the cover time (Ding~\cite{Ding}, Cortines, Louidor and Saglietti~\cite{CLS}) as well as the random walk on planar domains at times of order of the cover time (Abe and Biskup~\cite{AB} and Abe, Biskup and Lee~\cite{ABL}).  

While it is clear that the coupling between the local time~$L_t$ and the GFF~$h$ cannot be tight at the levels of the tree close to the root, this is quite different at (and near) the leaves where the local time is large. This suggests that the connection of~$L_t$ and~$h$ might remain strong on the range of the walk. We believe that this proposition warrants further study. The tree geometry may be a perfect setting for this.

\section{Tightness of maximal local time}
\label{sec-2}\noindent
We are now ready to commence the proofs. We start by showing tightness of centered maximal local time for the random walk started from the root and parametrized by the time spent there. Apart from being of independent interest for its uniformity in~$t>0$, tightness serves as a technical input in some of the proofs of our main theorems later.

\subsection{Main statement and preliminaries}
For each integer~$n\ge1$ and real $t>0$, consider the quantity
\begin{equation}
\label{E:1.10a}
a_n(t):=n\slb -\frac3{4\slb}\,\log n-\frac1{4\slb}\log\Bigl(\frac{n+\sqrt{t}}{\sqrt{t}}\Bigr).
\end{equation}
This is the aforementioned $t$-dependent centering sequence discovered in Abe~\cite{A18} which interpolates between the centering sequences in \eqref{E:2.3} and~\eqref{E:1.12i} as~$t$ varies from~$1$ to infinity. Write
\begin{equation}
\label{E:2.2o}
\HH_{n,t}:=\Bigl\{\max_{x\in\BbbL_n}L_t(x)>0\Bigr\}
\end{equation}
for  the event that the walk has hit (and spent positive amount of time at) the leaves prior to accumulating total time~$t$ at the root. We then claim:

\begin{theorem}
\label{thm-4}
There are $c_1,c_2>0$ such that for all $n\ge1$, all~$t>0$ and all~$u\in[0,n]$,
\begin{equation}
\label{E:2.2i}
P^\varrho\biggl(\,\Bigl|\,\max_{x\in\BbbL_n}\sqrt{L_t(x)}- \sqrt{t} - a_n(t\vee1)\Bigr|>u\,\bigg|\, \HH_{n,t}\biggr)\le c_1\texte^{-c_2u}.
\end{equation}
In particular, for each~$t>0$, the family
\begin{equation}
\label{E:2.3i}
\biggl\{\,\text{\rm law of }\,\max_{x\in\BbbL_n}\sqrt{L_t(x)}-\Bigl(\slb\,n-\frac1{\slb}\log n\Bigr)\text{\rm\ under }P^\varrho(\,\cdot\,|\,\HH_{n,t})\biggr\}_{n\ge1}
\end{equation}
of probability measures on~$\R$ is tight.
\end{theorem}

The truncation of the argument of~$a_n$ in \eqref{E:2.2i} is  necessary  because $t\mapsto a_n(t)$ behaves poorly for small~$t$.
The restriction to~$u\le n$ in turn arises from the fact that $\max_{x\in\BbbL_n}\sqrt{L_t(x)}$ can be kept at order unity by forcing the walk to never come back to the leaves after hitting~$\BbbL_n$ for the first time (as required by conditioning on~$\HH_{n,t}$). For any fixed~$t$, this shows that the left hand side of \eqref{E:2.2i} is always at least $\texte^{-cn}$, for some~$t$-dependent~$c>0$. Since our prime desire is to work uniformly in~$t$ and~$n$, we limit~$u$ to values where other strategies are more relevant. 

The proof of Theorem~\ref{thm-4} is based on the observation that conditioning on the value and the location $z\in\BbbL_n$ of the  local-time  maximum  restrains the maxima in the subtrees ``hanging off'' the string of vertices on the unique path from the root to~$z$. Combined with the Markov property of the local time and explicit knowledge of the local time law on the path between~$z$ and the root, this offers a way to trade assumptions on the upper tails of the maximum for control of the lower tails, and \emph{vice versa}. 

Although this trade-off could (at least in principle) be used to build a completely self-contained proof of tightness, in all the cases where this strategy has been implemented --- e.g., the GFF on the tree and subsets of~$\Z^2$; see \cite[Lecture~8]{B-notes} --- the upper tail tightness comes from an independent calculation. Thanks to Abe~\cite{A18}, this applies also for the local time on the tree. Indeed, he showed:

\begin{lemma}
\label{lemma-2}
There is~$c > 0$ such that for all~$t>0$,~$u\ge0$ and $n \ge 1$,
\begin{equation}
\label{E:2.4u}
P^\varrho\Bigl(\,\max_{x\in\BbbL_n}\sqrt{L_t(x)}- \sqrt t- a_n(t)\ge u\Bigr)\le c (1+u) \texte^{-2u\slb} \,.
\end{equation}
\end{lemma}

\begin{proofsect}{Proof}
See Abe~\cite[Proposition 3.1]{A18}.
\end{proofsect}

The estimate \eqref{E:2.4u} will be useful throughout our entire paper.
As $P^\varrho(\HH_{n,t})$ is bounded away from zero uniformly in~$n\ge1$ and $t\ge1$, it already supplies one half of the statement in \eqref{E:2.2i} when~$t\ge1$. We will address the small-$t$ corrections to \eqref{E:2.4u} needed for \eqref{E:2.2i} in Lemma~\ref{lemma-2t} below; our more difficult task is to supply an argument for lower-tail tightness. A key starting point for this is the following uniform bound on the probability that the maximal local time is at least $\sqrt t+a_n(t)$.

\begin{proposition}
\label{prop-1}
We have
\begin{equation}
\label{E:1.2}
\inf_{t\ge1}\inf_{n\ge1} P^\varrho\Bigl(\,\max_{x\in\BbbL_n}\sqrt{L_t(x)}\ge \sqrt t + a_n(t)\Bigr)>0.
\end{equation}
\end{proposition}

The proof of \eqref{E:1.2}, which we address next, will require several ingredients and rather precise estimates, due to the fact that we  wish  to work uniformly in~$t\ge1$. First we note that we only need to produce a uniform lower bound \eqref{E:1.2} for $t\le n$. Indeed,  Abe~\cite{A18} has also proved:

\begin{lemma}
\label{lemma-2a}
There is~$\tilde c > 0$ and~$n_0\ge1$ such that for all $n\ge n_0$, $t\ge n$ and~$u\in[0,2\sqrt n]$,
\begin{equation}
\label{E:2.4w}
P^\varrho\Bigl(\,\max_{x\in\BbbL_n}\sqrt{L_t(x)}- \sqrt t- a_n(t)\ge u\Bigr)\ge \tilde c (1+u) \texte^{-2u\slb} \,.
\end{equation}
\end{lemma}

\begin{proofsect}{Proof}
See Abe~\cite[Proposition 3.1]{A18}.
\end{proofsect}

(For the cases $n< n_0$ we invoke the limit statement \eqref{E:1.13r} along with the unbounded nature of the law of~$L_t$ for any fixed~$t$ and~$n$.)
As noted earlier, our proofs also rely strongly on the fact that the local time~$L_t$ enjoys  a  Markov property:

\begin{lemma}
\label{lemma-M}
Let~$x_1,\dots,x_m\in\T_n$ be vertices such that the subtrees $\T^{(1)},\dots,\T^{(m)}$ of~$\T_n$ rooted at these vertices are vertex-disjoint. Write~$n_i$ for the depth of~$\T^{(i)}$ and denote
\begin{equation}
\V(x_1,\dots,x_m):=\{x_1,\dots,x_m\}\cup\Bigl(\T_n\smallsetminus\bigcup_{i=1}^m\T^{(i)}\Bigr).
\end{equation}
Then, for each~$t>0$, conditional on $\{L_t(x)\colon x\in\V(x_1,\dots,x_m)\}$, the families
\begin{equation}
\{L_t(x)\colon x\in\T^{(i)}\}_{i=1}^m
\end{equation}
are independent with the~$i$-th family distributed as $\{L_u(x)\colon x\in\T_{n_i}\}$ for~$u:=L_t(x_i)$.
\end{lemma}

\begin{proofsect}{Proof}
As shown in Ding~\cite[Lemma~2.6]{Ding}, the local time on a tree has the following recursive structure: Given any non-leaf vertex~$x\in\T_n$, write $x_1,\dots,x_b$ for its descendants and~$\T(x)$ for the subtree of~$\T_n$ rooted at~$x$. Then, conditional on $\{L_t(z)\colon x\in\T_n\smallsetminus\T(x)\}$, the law of $(L_t(x_1),\dots,L_t(x_b))$  is that of independent copies of $\sum_{i=1}^N U_i$, where $\{U_i\}_{i\ge1}$ are i.i.d.\ exponentials and~$N$ is an independent Poisson with parameter~$L_t(x)$. By induction (whose details we leave to the reader), this readily yields the claim.
\end{proofsect}

Another important ingredient for us is the  explicit description of the local time process along the line of vertices from a leaf to the root.

\begin{lemma}
\label{lemma-3}
Let~$x_0 :=\varrho,  x_1,\dots,x_n$ be a path from the root to a leaf $x_n\in\BbbL_n$ and, given~$t>0$, let~$\{Y_s\}_{s\ge0}$ be the $0$-dimensional Bessel process started at~$\sqrt{2t}$ for~$t>0$. Then
\begin{equation}
\Bigl(\sqrt{L_t(x_1)},\dots,\sqrt{L_t(x_n)}\Bigr)\laweq\Bigl(\tfrac{1}{\sqrt{2}}Y_1,\dots,\tfrac{1}{\sqrt{2}} Y_n \Bigr).
\end{equation}
In addition, denoting by $\LL_Y^y$ the law of~$\{Y_s\colon s\le u\}$ with $Y_0=y$ a.s.\ and by~$\LL_B^x$ the law of the Brownian motion $\{B_s\colon s\le u\}$ started at~$B_0=x$ a.s., then for each~$r>0$
\begin{equation}
\label{E:1.5}
\frac{\textd\LL_Y^r}{\textd\LL_B^r}=\sqrt{\frac{r}{B_u}} \exp \left\{- \frac{3}{8} \int_0^u \textd s\,\frac{1}{B_s^2} \right\}\quad\text{\rm on }\{\hat\tau_0>0\},
\end{equation}
where $\hat\tau_0:=\inf\{t\ge0\colon B_t=0\}$.
\end{lemma}

\begin{proofsect}{Proof}
See Belius, Rosen and Zeitouni~\cite[Lemma 3.1(e) and formula~(2.12)]{BRZ}.
\end{proofsect}

\subsection{Uniform lower bound}
Having dispensed with the preliminaries, we are now ready to start addressing the uniform lower bound in Proposition~\ref{prop-1}.
The proof will make use of a collection of numbers~$\{t_k\}_{k=0}^n$ depending on $n\ge1$ and $t\ge1$ that obey
\begin{equation}
\label{E:1.6}
\sqrt{t_k}+a_{n-k}(t_k)=\sqrt{t}+a_n(t),
\end{equation}
with the convention $a_0(t) := 0$ for all $t>0$. The next lemma shows that these are well defined and that $k\mapsto t_k$ grows approximately quadratically in~$k$.

\begin{lemma}
\label{lemma-4}
For each~$n\ge1$ and~$t\ge1$, there are unique $\{t_k\}_{k=0}^n\subseteq[1,\infty)$ satisfying~\eqref{E:1.6}. Moreover, $k\mapsto t_k$ is strictly increasing with $\sqrt{t_0} = \sqrt{t}$, $\sqrt{t_n} = \sqrt{t} + a_n(t)$ and
\begin{equation}
\label{E:1.7}
\sqrt{t_k} \geq \sqrt{t} + \frac{k}{n} a_n(t) - C \log \bigl(1+ k \wedge (n-k) \bigr), 
\quad  \  k=0, \dots, n \,,
\end{equation}
for some $C < \infty$ independent of~$n$.
\end{lemma}

\begin{proof}
Existence and uniqueness follows since $t\mapsto\sqrt t+a_{n}(t)$, resp., $n\mapsto a_n(t)$ are both strictly increasing on~$t\ge1$, resp., $n\ge1$. This implies existence and uniqueness of the solution to \eqref{E:1.6} as well as strict monotonicity of $k\mapsto t_k$. The solutions in the cases $k=0$ and $k=n$  are  checked by hand. 

To show~\eqref{E:1.7}, note that $t_k \geq t$ and the monotonicity of $s\mapsto a_n(s)$ yields $a_{n-k}(t_k) \geq a_{n-k}(t) \geq a_n(t) - \sqrt{\log b}\, k$, which  in view of~\eqref{E:1.6} gives the upper bound $\sqrt{t_k} \leq \sqrt{t} + \sqrt{\log b}\, k$. It follows that
\begin{equation}
\frac{\sqrt{t} + n}{\sqrt{t}}
\frac{\sqrt{t_k}}{\sqrt{t_k} + (n-k)} \leq 
\frac{\sqrt{t} + n}{\sqrt{t}}
\frac{\sqrt{t} + \sqrt{\log b}\, k} 
{\sqrt{t} + n} = 1 + \sqrt{\log b} \frac{k}{\sqrt{t}}
\end{equation}
and consequently that $a_n(t) - a_{n-k}(t_k)$ is at least
\begin{equation}
\label{E:1.9}
\sqrt{\log b}\, k - \frac{3}{4 \sqrt{\log b}} \log \Bigl(\frac{n}{n-k}\Bigr) - 
\frac{1}{4} \log \Bigl(\frac{k}{\sqrt{t}}\Bigr) - C  \,,
\end{equation}
for some $C < \infty$, uniformly in $t \geq 1$. 

In order to compare \eqref{E:1.9} with $\frac{k}{n} a_n(t)$, we note that
\begin{equation}
\label{E:1.10}
-1 \leq \log^+ s - \frac{s}{r} \log^+ r \leq 1 + \log^+ \bigl(s \wedge (r-s)\bigr) \,, 
\end{equation}
for all $0 \leq s \leq r$ (c.f. Lemma~3.3 in~\cite{CHL17}), where $\log^+s$ stands for $\log (s \vee 1)$. Using this for $r := (\sqrt{t}+n)/\sqrt{t}$ and $s := k/\sqrt{t}$ in the second inequality gives
\begin{equation}
\begin{split}
\log \Bigl(\frac{k}{\sqrt{t}}\Bigr)
& \leq 
\frac{k}{\sqrt{t}+n} \log \Bigl(\frac{\sqrt{t}+n}{\sqrt{t}}\Bigr) + 
1 + \log \Bigl(\frac{k}{\sqrt{t}} \wedge \frac{n-k+\sqrt{t}}{\sqrt{t}} \Bigr) \\
& \leq \frac{k}{n} \log \Bigl(\frac{\sqrt{t}+n}{\sqrt{t}}\Bigr) + 
2 + \log \bigl(k \wedge (n-k)\bigr) \,.
\end{split}
\end{equation}
On the other hand, plugging $s=n-k$ and $r=n$ in the first inequality in~\eqref{E:1.10} gives
\begin{equation}
\log \Bigl(\frac{n}{n-k}\Bigr) \leq \frac{k}{n} \log n +1 
\end{equation} 
In view of~\eqref{E:1.9} and the definitions of $a_n(t)$ and $\sqrt{t_k}$, this shows~\eqref{E:1.7}.
\end{proof}

Next we note that, thanks to the Markov property, it is sufficient to prove the lower bound in \eqref{E:1.2} just for~$t$ sufficiently large (albeit uniformly in~$n$). Indeed, we have:

\begin{lemma}
\label{lemma-2.6a}
For all $s,t>0$, the laws of $\{L_s(x)\colon x\in\T_n\smallsetminus\{\varrho\}\}$ and $\{L_t(x)\colon x\in\T_n\smallsetminus\{\varrho\}\}$ are mutually absolutely continuous. More precisely, for any $s,t>0$ and any Borel~$\EE\subseteq\R^{\T_n}$ that does not depend on the coordinate at~$\varrho$,
\begin{equation}
\label{E:2.9o}
P^\varrho(L_t\in\EE)\ge c(s,t) P^\varrho(L_s\in\EE)^2,
\end{equation}
where $c(s,t):=\texte^{(2s-t-s^2)b/t}$.
\end{lemma}

\begin{proofsect}{Proof}
Let~$N_t$, resp., $N_s$ Poisson random variables with parameters~$t$, resp.,~$s$ and let $f_{s,t}(n):=P(N_s=n)/P(N_t=n)=(s/t)^n\texte^{t-s}$ be the Radon-Nikodym derivative of their laws. We will now construct~$L_s$, resp.,~$L_t$ as follows: First we use the recursive structure of the local time to realize the local time at the descendants of~$\varrho$ via independent Poisson random variables $N_s^{(1)},\dots,N_s^{(b)}$, resp., $N_t^{(1)},\dots,N_t^{(b)}$ as detailed in the proof of Lemma~\ref{lemma-M} and then invoke the Markov property to generate the local time in the rest of the tree. For any Borel~$\EE\subseteq\R^{\T_n}$ not depending on the value at~$\varrho$ we have
\begin{equation}
P^\varrho(L_s\in\EE) = E^\varrho\Bigl(1_{\{L_t\in\EE\}}\prod_{i=1}^b f_{s,t}(N_t^{(i)})\Bigr).
\end{equation}
The Cauchy-Schwarz inequality then gives \eqref{E:2.9o} with~$c:=[E(f_{s,t}(N_t)^2)]^{-b}$. A calculation shows that $E(f_{s,t}(N_t)^2) = \texte^{t-2s+s^2/t}$.
\end{proofsect}

The proof of Proposition~\ref{prop-1} opens up by a calculation that converts the probability that the maximum occurs at a given vertex to a ``barrier estimate'' for the local-time profile along the path from the root to that vertex.

\begin{lemma}
\label{lemma-2.7a}
For all $A>0$ there is $\tilde c(A)\in[0,1]$ with $\tilde c(A)\to1$ as $A\to\infty$ such that the following holds for all $t\ge1$ and all $n\ge1$: Given $x_n\in\BbbL_n$, and writing~$(x_0,\dots,x_n)$ for the vertices on the unique path from the root to~$x_n$, set
\begin{multline}
\label{E:2.17u}
\quad
\EE_n:=\bigcap_{k=1}^{n-1}\biggl\{\sqrt{L_t(x_k)} 
< \sqrt{t_k} - A \log\bigl(1+k\wedge (n-k) \bigr)
\\+\frac kn\Bigl(\sqrt{L_t(x_n)}-\sqrt t-a_n(t)\Bigr)
\biggr\}.
\quad
\end{multline}
Then
\begin{multline}
\label{E:2.19}
\quad
P^\varrho\biggl(\sqrt{L_t(x_n)}=\max_{x\in\BbbL_n}\sqrt{L_t(x)}\ge \sqrt t+a_n(t)\biggr)
\\
\ge 
\tilde c(A) P^\varrho\biggl(\,\EE_n \cap\Bigl\{ \sqrt{L_t(x_n)}\ge\sqrt t+a_n(t)\Bigr\}\biggr)
\quad 
\end{multline}
\end{lemma}

\begin{proofsect}{Proof}
Fix~$x_n\in\BbbL_n$ and let~$(x_0,\dots,x_n)$ be the unique path from the root to~$x_n$. Removing the edges along this path splits $\T_n$ into a collection of disjoint subtrees rooted at the vertices on the path. Writing $\BbbL'_{n-k}(x_k)$ for the leaves of the subtree whose root is at~$x_k$ we then have
\begin{equation}
\label{E:1.8}
\begin{aligned}
P^\varrho\Bigl(&\sqrt{L_t(x_n)}=\max_{x\in\BbbL_n}\sqrt{L_t(x)}\ge \sqrt t+a_n(t)\Bigr)
\\
&=\int_{s\ge \sqrt t+a_n(t)} P^\varrho\Bigl(\,\max_{x\in\BbbL_n}\sqrt{L_t(x)}\le s,\,\sqrt{L_t(x_n)}\in\textd s\Bigr)
\\
&\ge\int_{s\ge \sqrt t+a_n(t)} P^\varrho\biggl(\,\EE_n\cap\bigcap_{k=0}^{n-1}\Bigl\{\max_{x\in\BbbL'_{n-k}(x_k)}\sqrt{L_t(x)} < s\Bigr\},\,\sqrt{L_t(x_n)}\in\textd s\biggr),
\end{aligned}
\end{equation}
where we also noted that the maximal local time occurs at a unique vertex almost surely.
 
Next we note that, by the Markov property of the local time (cf Lemma~\ref{lemma-M}), conditional on the local time at $x_0,\dots,x_n$, the maxima $\{\max_{x\in\BbbL'_{n-k}(x_k)}\sqrt{L_t(x)}\}_{k=0}^n$ are independent with law depending only on the value of the local time at the root vertex of the corresponding subtree. Using that $\BbbL'_{n-k}(x_k)$ is a (proper) subset of the leaves of a regular tree of depth~$n-k$, Lemma~\ref{lemma-M} shows that, for all~$r,u,\tilde u\in\R$ with $u\ge \tilde u$ and $\sqrt{t_k}-r+\frac kn \tilde u>0$,
\begin{multline}
\label{E:2.19s}
\quad P^\varrho\biggl(\,\max_{x\in\BbbL'_{n-k}(x_k)}\sqrt{L_t(x)}< \sqrt{t}+a_n(t)+u\,\bigg|\,\sqrt{L_t(x_k)}=\sqrt{t_k}-r+\frac kn \tilde u\biggr)
\\\ge
P^\varrho\Bigl(\,\max_{x\in\BbbL_{n-k}}\sqrt{L_{t_k'}(x)}<\sqrt{t}+a_n(t)+u\Bigr)
\quad
\end{multline}
for
\begin{equation}
t_k':=\begin{cases}
(\sqrt t_k-r+\tfrac kn u)^2,\qquad&\text{if }r\ge\frac kn u,
\\
t_k,\qquad&\text{else},
\end{cases}
\end{equation}
where we also used that $s\mapsto L_s$ is increasing. Since $t_k'\le t_k$, the identity \eqref{E:1.6} along with the upward monotonicity of $s \mapsto a_{n-k}(s)$ show
\begin{equation}
\sqrt{t}+a_n(t)+u\ge \sqrt{t_k'}+a_{n-k}(t_k')+r.
\end{equation}
Setting $c':=\sup_{r\ge0}c(1+r)\texte^{-r\slb }$ for~$c$ is as in \eqref{E:2.4u}, Lemma~\ref{lemma-2} then gives
\begin{equation}
P^\varrho\Bigl(\,\max_{x\in\BbbL_{n-k}}\sqrt{L_{t_k'}(x)}<\sqrt{t}+a_n(t)+u\Bigr)
\ge 1- c'\texte^{-r\slb}
\end{equation}
for all $r,u\ge0$ satisfying $\sqrt{t_k}-r+\frac kn u>0$.

Using the above for the choices $r:=A\log(1+k\wedge(n-k))$, $u:=\sqrt{L_t(x_n)}-\sqrt t-a_n(t)$ and $\tilde u \le u$ in the subtree rooted at~$x_k$, the restrictions imposed by~$\EE_n$ along with the aforementioned conditional independence yield
\begin{multline}
\label{E:2.18u}
P^\varrho\biggl(\,\EE_n\cap\bigcap_{k=0}^{n-1}\Bigl\{\max_{x\in\BbbL'_{n-k}(x_k)}\sqrt{L_t(x)}<\sqrt{t}+a_n(t)+s\Bigr\} \,\bigg| \,\sqrt{L_t(x_n)}-\sqrt t-a_n(t) = s\biggr)
\\
\ge 
\tilde c_n(A)P^\varrho\Bigl(\,\EE_n \,\Big|\, \sqrt{L_t(x_n)}-\sqrt t-a_n(t) = s \Bigr),
\end{multline}
where
\begin{equation}
\tilde c_n(A):=\prod_{k=1}^{n-1}\biggl(1-c'(1+k\wedge(n-k))^{-A\slb}\biggr)
\end{equation}
where we also noted that $\sqrt{t_k}-r+\frac kn u \le 0$ for some $k=1,\dots,n$ implies that both sides of~\eqref{E:2.18u} are zero.
Writing $\tilde c(A):=\max\{0,\inf_{n\ge1}\tilde c_n(A)\}$, the claim follows by noting that $n\mapsto\tilde c_n(A)$ is convergent once~$A$ is sufficiently large.
\end{proofsect}

We are now finally in a position to give:

\begin{proofsect}{Proof of Proposition~\ref{prop-1}}
Fix~$A>0$ be so large that~$\tilde c(A)$ in Lemma~\ref{lemma-2.7a} obeys~$\tilde c(A)>0$. Recall the event~$\EE_n$ from \eqref{E:2.17u}. It suffices to show that for all~$n\ge1$ and all $t\le n$,
\begin{equation}
\label{E:2.21}
P^\varrho\biggl(\,\EE_n \cap\Bigl\{ \sqrt{L_t(x_n)}\ge\sqrt t+a_n(t)\Bigr\}\biggr)\ge\hat c(t)b^{-n},
\end{equation}
where~$t\mapsto\hat c(t)$ is uniformly positive for~$t$ sufficiently large. Indeed, \eqref{E:2.21} lower bounds the probability on the left of \eqref{E:2.19} by $\tilde c(A)\hat c(t)b^{-n}$. Since that the maximum is a.s.\ achieved at a unique leaf, summing this over~$x_n\in\BbbL_n$ bounds the probability in \eqref{E:1.2} by $\tilde c(A)\hat c(t)$ from below. Lemmas~\ref{lemma-2a} and~\ref{lemma-2.6a} then complete the claim.

We thus have to prove \eqref{E:2.21}. Write~$\{Y_s\}_{s\ge0}$ for the $0$-dimensional Bessel process and denote the law with initial value~$\sqrt{2t}$ by~$P^{\sqrt{2t}}$. Pick $\eta\in(0,\slb)$, let $A':=A+C$ for~$C$ the constant from Lemma~\ref{lemma-4} and define the events
\begin{multline}
\quad
\FF_n:=\biggl\{\tfrac1{\sqrt2} Y_s \le \sqrt{t}+\frac{s}{n} a_n(t)
 - A'\log \bigl(1 + s \wedge (n-s)\bigr)
 \\
 + \frac sn\bigl(\tfrac1{\sqrt2}Y_n-\sqrt t-a_n(t)\bigr)\,\colon\,s\in[1,n-1]
\biggr\} \quad
\end{multline}
and
\begin{equation}
\GG_n:=\Bigl\{\tfrac1{\sqrt2} Y_s \ge \frac12\sqrt{t}+\eta s\colon\, s\in[0,n]
\Bigr\}.
\end{equation}
Using the bound in Lemma~\ref{lemma-4} for~$\sqrt{t_k}$ in \eqref{E:2.17u} and invoking Lemma~\ref{lemma-3} then shows that the probability in \eqref{E:2.21} is at least
\begin{equation}
\label{E:1.17.1}
P^{\sqrt{2t}}\biggl(\Bigl\{0\le\tfrac1{\sqrt2} Y_n - \sqrt t-a_n(t)\le\sqrt n\Bigr\}\cap \FF_n\cap \GG_n\biggr),
\end{equation}
where the upper bound by~$\sqrt n$ has been introduced for later convenience.

Next note that $Y_s\ge\frac1{\sqrt2}(\sqrt t+2\eta s)$ on~$\GG_n$ and so $\int_0^n Y_s^{-2}\textd s\le 2\int_0^\infty (\sqrt t+2\eta s)^{-2}\textd s$ uniformly in~$n\ge1$.
Invoking the second part of Lemma~\ref{lemma-3}, we may replace the $0$-dimensional Bessel process by standard Brownian motion $\{B_s\}_{s\ge0}$ and bound \eqref{E:1.17.1} from below by
\begin{multline}
\label{E:2.28i}
\texte^{-\frac34\int_0^\infty\frac{\textd s}{(\sqrt t+2\eta s)^2}} \biggl(\frac{\sqrt{t}}{\sqrt t+ a_n(t)+\sqrt n}\biggr)^{1/2}
\\
\times P^{\sqrt{2t}}\biggl(\Bigl\{0\le\tfrac1{\sqrt2} B_n - \sqrt t-a_n(t)\le\sqrt n\Bigr\}\cap \wt\FF_n\cap \wt\GG_n\biggr),
\end{multline}
where~$\wt\FF_n$ and~$\wt\GG_n$ are the events~$\FF_n$ and~$\GG_n$ above for~$Y_s$ replaced by~$B_s$ and where the prefactors are uniform lower bounds on the Radon-Nikodym derivative in \eqref{E:1.5}.

The probability on the right of \eqref{E:2.28i} will be handled by conditioning on~$B_n$. Define functions $f,g_t\colon[0,n]\to\R$ by
\begin{equation}
f(s):=\sqrt{2} A'\log\bigl(1+s\wedge(n-s)\bigr)\quad\text{and}\quad g_t(s):=\sqrt{\frac t2}-\sqrt2\,\tilde\eta(t) s,
\end{equation}
where $\tilde\eta(t):=\eta-\inf_{n\ge1}\frac{a_n(t)}n$.
Since the Brownian bridge starting from $\sqrt{2t}$ at time~$0$ and terminating at $\sqrt2(\sqrt t+a_n(t)+u)$ at time~$n$ has the same law as the sum of $s\mapsto \sqrt{2t}+\sqrt{2}\frac sn (a_n(t)+u)$ and the Brownian bridge from~$0$ at time~$0$ to~$0$ at time~$n$, a calculation shows that, for all~$u\ge 0$,
\begin{multline}
\label{E:2.36}
\quad
P^{\sqrt{2t}}\Bigl(\wt\FF_n\cap\wt\GG_n\,\Big|\,\tfrac1{\sqrt2} B_n = \sqrt t+a_n(t)+u\Bigr)
\\
\ge
P^0\Bigl( B\le  - f\text{ on }[1,n-1]\,\,\wedge\,\, B\ge  -  g_t\text{ on }[0,n]\,\Big|\,B_n=0\Bigr).
\quad
\end{multline}
Now observe that $t\mapsto\tilde\eta(t)$ is non-increasing with
\begin{equation}
\tilde\eta(t)\,\underset{t\to\infty}\longrightarrow\,\eta-\slb +\frac3{4\slb}\,\sup_{n\ge1}\frac{\log n}n \,.
\end{equation}
It follows that $\tilde\eta(t)$ is negative for all~$t$ large as soon as~$\eta<0.6\slb$. With this choice~$s\mapsto g_t(s)$  increases  linearly and so a standard Ballot estimate (cf Lemma~\ref{lemma-A.1}) shows that the probability on the right is at least $c'(t)/n$, where~$t\mapsto c'(t)$ is non-decreasing and strictly positive for~$t$ sufficiently large. This bounds the probability in \eqref{E:2.28i} from below by
\begin{equation}
\label{E:2.36i}
\frac{c'(t)}n
 \,P^{\sqrt{2t}}\Bigl(0\le\tfrac1{\sqrt2} B_n - \sqrt t-a_n(t)\le\sqrt n\Bigr)
\end{equation}
uniformly in~$n\ge1$ and~$t\ge1$. 

It remains to estimate the probability in \eqref{E:2.36i} which, as noted earlier, we need to do just for~$t\in[1,n]$. Noting that $\frac1{\sqrt2}B_n$ has mean~$\sqrt t$ and variance~$n/2$ under~$P^{\sqrt{2t}}$, a calculation shows
\begin{equation}
P^{\sqrt{2t}}\Bigl(0\le\tfrac1{\sqrt2} B_n - \sqrt t-a_n(t)\le\sqrt n\Bigr)
\ge c''(t)\frac{\sqrt n}{a_n(t)}\,\texte^{-\frac{a_n(t)^2}{n}},
\end{equation}
where $c''(t):=\frac12\pi^{-1/2}\texte^{-1}(1-\texte^{-q(t)})$ for $q(t):=\inf_{n\ge1}\sqrt 2 a_n(t)/\sqrt n$. Since~$t\mapsto q(t)$ is uniformly positive for~$t$ sufficiently large, so is~$t\mapsto c''(t)$. To bound the right-hand side note that
\begin{equation}
a_n(t)\le\slb\, n-\frac1{\slb}\log n+\frac1{4\slb}\log\sqrt t
\end{equation}
and so
\begin{equation}
\frac{a_n(t)^2}n\le n\log b - 2\log n +\frac12\log\sqrt t +\frac1{\log b}\frac{\tfrac{1}{16}(\log\sqrt t)^2
 + (\log n)^2 }n.
\end{equation}
The last term is at most~$1$ for all $t\in[1,n]$ and all $b\ge2$ and so, combining the above estimates, the probability in \eqref{E:2.21} is at least
\begin{equation}
c_0 \biggl(\frac{\sqrt{t}}{\sqrt t+ a_n(t)+\sqrt n}\biggr)^{1/2} \,\frac{c'(t)}n\,c''(t)\frac{\sqrt n}{a_n(t)} \frac{\texte^{-1}n^{2}b^{-n}}{(\sqrt t)^{1/2}},
\end{equation}
where~$c_0$ is a shorthand for the exponential prefactor in \eqref{E:2.28i}. Using that~$t\le n$ and $a_n(t)\le n\slb$ we then get \eqref{E:2.21}.
\end{proofsect}


\subsection{Proof of uniform tightness}
Having settled Proposition~\ref{prop-1}, we move to the proof of Theorem~\ref{thm-4}. Our first item of business is the upper bound for the conditional probability \eqref{E:2.2i} in the regime of small~$t$. We start by giving  a  small-$t$ asymptotic for  the probability of the conditional event.

\begin{lemma}
\label{lemma-2.10H}
For all $n\ge1$ and all~$t\ge0$,
\begin{equation}
b t\ge P^\varrho(\HH_{n,t})
\ge \frac{b-1}b (1-\texte^{-bt}).
\end{equation}
\end{lemma}

\begin{proofsect}{Proof}
Let~$T$ be the time of the first jump of the walk and let~$\theta_t$ denote the shift by~$t$ on the path space of the walk. Write $\tau_{\BbbL_k}$ for the first hitting time of level~$\BbbL_k$. 
Then, for the walk started at~$\varrho$,
\begin{equation}
\{T\le t\}\supseteq \HH_{n,t}\supseteq\{T<t\}\cap\theta_T^{-1}\bigl(\{\tau_{\BbbL_n}<\tau_\varrho\}\bigr).
\end{equation}
Noting that~$T$ is exponential with parameter~$b$, the claim follows from the strong Markov property along with the fact the probability that a random walk on an infinite~$b$-ary tree started from a neighbor of the root never hits the root equals $\frac{b-1}b$.
\end{proofsect}

Lemma~\ref{lemma-2.10H} allows us to upgrade Abe's uniform upper bound in Lemma~\ref{lemma-2} to include conditioning on~$\HH_{n,t}$ when~$t$ is small.

\begin{lemma}
\label{lemma-2t}
There is~$c'>0$ such that for all~$t\in(0,1]$,~$u\ge0$ and $n \ge 1$,
\begin{equation}
\label{E:2.4q}
P^\varrho\Bigl(\,\max_{x\in\BbbL_n}\sqrt{L_t(x)}- \sqrt{1}- a_n(1)\ge u\,\Big|\,\HH_{n,t}\Bigr)\le c' (1+u) \texte^{-2u\slb} \,.
\end{equation}
\end{lemma}

\begin{proofsect}{Proof}
We may assume that~$u$ is so large that the right-hand side of \eqref{E:2.4u} is less than~$1/2$. The Markov property of the random walk parametrized by the time spent at the root tells us that, for any integer $k\ge1$, any reals $t_1,\dots,t_k>0$ and with $L_{t_1}^{(1)},\dots,L_{t_k}^{(k)}$ denoting independent samples from the local time with time parameters $t_1,\dots,t_k$,
\begin{equation}
L_{t_1+\dots+t_k} \text{ under }P^\varrho \,\,\laweq \,\,L_{t_1}^{(1)}+\dots+L_{t_k}^{(k)}\text{ under }(P^\varrho)^{\otimes k}.
\end{equation}
The sum on the right being smaller than a constant implies that each term is smaller than that constant. Writing~$\EE_t(u)$ for the event in \eqref{E:2.4q}, this shows
\begin{equation}
1-P^\varrho\bigl(\EE_1(u)\bigr)\le \bigl[1-P^\varrho\bigl(\EE_{1/k}(u)\bigr)\bigr]^k.
\end{equation}
Since $P^\varrho(\EE_1(u))<1/2$ by our assumption on~$u$, the inequalities $\log(1-x)\ge-2x$ for $x\in(0,1/2)$ and $1-x\le\texte^{-x}$ for all~$x$ yield
\begin{equation}
P^\varrho\bigl(\EE_{1/k}(u)\bigr)\le\frac2k P^\varrho\bigl(\EE_1(u)\bigr).
\end{equation}
For $k:=\lfloor 1/t\rfloor$, which entails $t\le\frac1k\le 2t$, we then get $P^\varrho(\EE_{t}(u))\le 4tP^\varrho(\EE_1(u))$ thanks to the monotonicity of $t\mapsto\EE_t(u)$. The claim follows from Lemmas~\ref{lemma-2} and~\ref{lemma-2.10H}.
\end{proofsect}

A majority of our effort throughout the rest of this subsection will be spent on upgrading the uniform lower bound \eqref{E:1.2} to a lower-tail estimate for the maximal local time. The precise statement is the content of:

\begin{proposition}
\label{prop-2}
There are $\alpha_1,\alpha_2>0$ such that for all $n\ge1$, all~$t>0$ and all~$u\in[0,n]$,
\begin{equation}
\label{E:2.43}
P^\varrho\Bigl(\,\max_{x\in\BbbL_n}\sqrt{L_t(x)}- \sqrt{t} - a_n(t\vee1)<-u\,\Big|\, \HH_{n,t}\Bigr)\le \alpha_1\texte^{-\alpha_2 u}.
\end{equation}
\end{proposition}

Indeed, with \eqref{E:2.43} in hand we readily conclude:

\begin{proofsect}{Proof of Theorem~\ref{thm-4} from Proposition~\ref{prop-2}}
The bound \eqref{E:2.2i} follows by combining the statements of Proposition~\ref{prop-1}, Lemma~\ref{lemma-2t} and Proposition~\ref{prop-2}. For \eqref{E:2.3i} we note that, for each~$t>0$ fixed,
\begin{equation}
\sqrt{t}+a_n(t\vee1)=\slb\, n-\frac1{\slb}\log n+O(1)
\end{equation}
with $O(1)$ bounded as~$n\to\infty$ for any~$t>0$.
\end{proofsect}

It remains to construct a proof of Proposition~\ref{prop-2}. The argument is unfortunately somewhat complicated due to the need to distinguish two regimes depending, roughly, on whether~$u$ is smaller or larger than~$\sqrt t$. The latter case  must also reflect on the fact that the probability of the conditional event vanishes as $t\downarrow0$ and that our uniform estimate in Proposition~\ref{prop-1} only applies to~$t\ge1$. 

The strategy is nonetheless similar in both regimes: We first prove that, up to a probability that is exponentially small in~$u$, the local time accumulated after the walk has hit the leaves for the first time exceeds some~$s_k\ge1$ at more than~$r(k)$ vertices of~$\BbbL_k$, for some~$k$ satisfying $1\le k\le n$. If this works for an $s_k$ with
\begin{equation}
\label{E:2.44}
\sqrt{t}+a_{n}(t\vee1)-u
\le \sqrt{s_k}+a_{n-k}(s_k),
\end{equation}
then, by Proposition~\ref{prop-1} and the Markov property in Lemma~\ref{lemma-M}, conditional on the aforementioned event, the probability in \eqref{E:2.43} is at most $(1-q)^{r(k)}$, where~$q$ is the infimum in \eqref{E:1.2}. The key problem is thus to ensure the validity of~\eqref{E:2.44} along with $r(k)$ being at least a constant times~$u$, which we need to get exponential decay in~$u$.

\smallskip
We first address the regime where~$u/\sqrt t$ is small because it is considerably simpler. Here the parameters will be chosen in such a way that all of the vertices at depth~$k$ will carry a sufficiently large value of the local time. This relies on:

\begin{lemma}
\label{lemma-2.13}
For each~$\beta>\sqrt{\log b}$, $k\ge1$ and $t>4(\beta k)^2$,
\begin{equation}
\label{E:2.53}
P^\varrho\Bigl(\,\min_{x\in\BbbL_k}L_t(x)< (\sqrt t-2\beta k)^2\Bigr)\le 2\texte^{-(\beta^2-\log b) k}.
\end{equation}
\end{lemma}

\begin{proofsect}{Proof}
We will make a convenient (and singular) use of the Second Ray-Knight Theorem of Eisenbaum, Kaspi, Marcus, Rosen and Shi~\cite{EKMRS} which in the version of Zhai~\cite{Zhai} says that there is a coupling of~$L_t$ with two copies~$h$ and~$\tilde h$ of the BRW with step distribution~$\NN(0,1/2)$ such that~$L_t$ is independent of~$h$ and, for all $x\in\T_n$,
\begin{equation}
\label{E:2.54}
L_t(x)+h_x^2 = (\tilde h_x+\sqrt t)^2.
\end{equation}
If $\max_{x\in\BbbL_k}|h_x|\le \beta k$ and $\max_{x\in\BbbL_k}|\tilde h_x|\le \beta k$ for some~$\beta>0$ with~$\sqrt t\ge 2\beta k$, then \eqref{E:2.54} forces $\min_{x\in\BbbL_k} L_t(x)\ge (\sqrt t-2\beta k)^2$. It follows that the probability in \eqref{E:2.53} is at most twice that of~$\max_{x\in\BbbL_k}|h_x|>\beta k$.  
A routine first moment estimate now yields \eqref{E:2.53}.
\end{proofsect}

With this in hand, we are ready to give:

\begin{proofsect}{Proof of Proposition~\ref{prop-2} for $u\le 2\sqrt t$}
As the probability in the statement is non-increasing in~$u$, we may replace~$u$ by quarter thereof and assume that~$u<\frac12\sqrt t$. Adjusting the constants in the statement allows us to suppose that~$u$, and thus also~$t$, is larger than any prescribed constant. In particular, we may and will assume that~$t\ge1$. Recall also that~$u\le n$, which means that also~$n$ may  be assumed large.

Fix $\beta>\sqrt{\log b}$. Since~$u$ (and thus~$t$) is large, there is an integer~$k\ge1$ for which~$s_k$ defined by
\begin{equation}
\sqrt{s_k}:=\sqrt{t-\sqrt t}-\beta k
\end{equation}
is meaningful and 
\begin{equation}
\label{E:2.56}
k\slb \le u-\sqrt{t}+\sqrt{s_k}+\frac1{8\slb}\log\frac t{s_k}
\end{equation}
holds. Taking~$k$ to be the largest integer with this property, the assumed relation between~$u$ and~$\sqrt t$ implies $u=(\beta+\slb)k+O(1)$. The condition \eqref{E:2.56} then ensures \eqref{E:2.44} and, since $\beta+\sqrt{\log b}>1$  and $u\le n$,  we have~$k<n$.

Recall that $\tau_{\BbbL_k}$ denotes the first hitting time of level~$\BbbL_k$ and~$\wt\tau_\varrho(s)$ stands for the first time the (actual) time at the root reaches~$s$. Using the notation~$\theta_t$ for the shift on the path space, the desired event is then contained in \begin{equation}
\label{E:2.55}
\bigl\{\tau_{\BbbL_n}>\wt\tau_\varrho(\sqrt t)\bigr\}\cup\theta_{\wt\tau_\varrho(\sqrt t)}^{-1}\biggl(\Bigl\{\max_{x\in\BbbL_n}\sqrt{L_{t-\sqrt t}(x)}< \sqrt t + a_n(t)-u\Bigr\}\biggr).
\end{equation}
The probability of the first event in \eqref{E:2.55} is exponentially small in~$\sqrt t$ which by the assumed relation between~$u$ and~$t$ is exponentially small in~$u$. For the second event in \eqref{E:2.55}, we first use the strong Markov property to drop the shift and then use Lemma~\ref{lemma-2.13} to remove the event in \eqref{E:2.53} with~$t$ replaced by~$t-\sqrt t$. This is permitted thanks to the assumption $u<\frac12\sqrt t$ which yields
\begin{equation}
\label{E:2.57}
\sqrt{t-\sqrt t}-2\beta k\ge \sqrt{t}-\frac{2\beta}{\beta+\slb}u+O(1)\ge1
\end{equation}
once~$u$ is sufficiently large. In light of \eqref{E:2.44} and the Markov property, the probability that $\sqrt{L_{t-\sqrt t}}\le\sqrt t+a_n(t)-u$ holds everywhere on~$\BbbL_n$ while $\sqrt{L_{t-\sqrt t}}$ exceeds the quantity on the left of \eqref{E:2.57} everywhere on~$\BbbL_k$ is at most $(1-q)^{b^k}$, where~$q$ is the double infimum in \eqref{E:1.2}. Since~$k$ is proportional to~$u$, this implies the claim.
\end{proofsect}

Next we move to the regime where~$u$ is larger than~$\sqrt t$, which includes the subtle case of $t\in(0,1)$. Recall the notation~$\ell_s(x)$ from \eqref{E:1.1q} for the actual time the path~$\{X_u\}_{0\le u\le s}$ spends at~$x$ and~$\tau_\varrho$ from \eqref{E:1.2q} for the first time the walk is at the root.

\begin{lemma}
\label{lemma-1.5}
There are~$\tilde \alpha_1,\tilde \alpha_2>0$ such that for each~$k=1,\dots n$ and each~$z\in\BbbL_k$,
\begin{equation}
\label{E:1.22}
P^z\Bigl(\sum_{x\in\BbbL_k}1_{\{\ell_{\tau_\varrho}(x)\ge1\}}<k\Bigr)\le \tilde \alpha_1\texte^{-\tilde \alpha_2k}.
\end{equation}
\end{lemma}

\begin{proofsect}{Proof}
The estimate consists of three steps that depend on naturals $\ell$, $m$ and $r$ whose specific values will be identified at the very end.
Pick~$z\in\BbbL_k$ and consider the Markov chain~$X$ started from~$z$ and observed until~$\tau_\varrho$.
Denote by~$R_\ell$ the number of the excursions of~$X$ that start at level~$k$, reach level~$\ell\in\{1,\dots,k-1\}$ and then return to level~$k$ before hitting the root. A harmonic function calculation shows that an excursion that has reached level~$\ell$ will return to level~$k$ before hitting the root with probability~$p_{k,\ell}:=\sum_{j=0}^{\ell-1}b^{-j}/\sum_{j=0}^{k-1}b^{-j}$. By the strong Markov property,~$R_\ell+1$ is Geometric with parameter~$1-p_{k,\ell}$. As $1-p_{k,\ell}\le b^{-\ell}$, we have
\begin{equation}
\label{E:2.45}
P^z\bigl(R_\ell< m\bigr)\le m b^{-\ell}
\end{equation}
holds for all~$m\ge1$ and all~$\ell\in\{1,\dots,k-1\}$.

Each excursion that reaches level~$\ell$ and returns to level~$k$ before reaching the root hits~$\BbbL_k$ at any of~$b^{k-\ell}$ descendants of the last point visited at level~$\ell$ equally likely. Ignoring that different last points at level~$\ell$ will lead to different sets of descendants at level~$k$ leads to the bound
\begin{equation}
\label{E:2.46}
P^z\Bigl(\,R_\ell\ge m,\,\sum_{x\in\BbbL_k}1_{\{\tau_x<\tau_\varrho\}}\le r\Bigr)
\le \binom{b^{k-\ell}}r\,\Bigl(\frac r{b^{k-\ell}}\Bigr)^m,
\end{equation}
where the binomial coefficient counts the number of ways to choose a set of $b^{k-\ell}-r$ unvisited points and $r/b^{k-\ell}$ bounds the probability that an excursion will avoid this set. 

The waiting times of the walk at the first hitting point at level~$k$ are exponential with mean at least~$1/b$ (the mean is~$1$ when~$k=n$). For each natural~$q\ge1$, the union bound shows
\begin{equation}
\label{E:2.47}
P^z\Bigl(\,\sum_{x\in\BbbL_k}1_{\{\tau_x<\tau_\varrho\}}>r,\,\sum_{x\in\BbbL_k}1_{\{\ell_{\tau_\varrho}(x)\ge1\}}<q\Bigr)\le\binom rq (1-\texte^{-b})^q,
\end{equation}
where the binomial coefficient expresses the number of ways to choose a set of size~$q$ from~$r$ vertices (namely, those satisfying $\tau_x<\tau_\varrho$) and $1-\texte^{-b}$ bounds the probability that the waiting time of the walk during the first visit to that vertex is less than one. 

In order to bound the right-hand sides above, note that $\binom rk\le (\frac rk)^k(\frac r{r-k})^{r-k}$. For $r:=\beta k$ with~$\beta>1$ such that~$\beta k$ is an integer, the right-hand side of \eqref{E:2.47} is bounded as
\begin{equation}
\label{E:2.48}
\binom{\beta k}k\le\Bigl[\beta\Bigl(\frac\beta{\beta-1}\Bigr)^{\beta-1}(1-\texte^{-b})\Bigr]^k.
\end{equation}
The term in the large square brackets tends to $1-\texte^{-b}$ as~$\beta\downarrow1$, so a choice of~$\beta>1$ can be made (with~$\beta k$ integer) for each~$k$ large enough to make \eqref{E:2.48} decay exponentially in~$k$ with a uniform rate. With this choice we now put $m:=2r$ and~$\ell:=\lfloor k/2\rfloor$. Then \eqref{E:2.45} decays exponentially in~$k$ and \eqref{E:2.46} even exponentially in~$k\log k$.
\end{proofsect}

We will also need:

\begin{lemma}
\label{lemma-2.12}
For any~$k=1,\dots,n-1$, any~$z\in\BbbL_k$ and with~$\scrF_k:=\sigma(\ell_{\tau_\varrho}(x)\colon x\in\BbbL_k)$,
\begin{equation}
P^z\Bigl(\max_{x\in\BbbL_n}\sqrt{\ell_{\tau_\varrho}(x)}\le\cdot\,\Big|\,\scrF_k\Bigr)
=\prod_{y\in\BbbL_k} P^\varrho\Bigl(\,\max_{x\in\BbbL_{n-k}}\sqrt{L_s(x)}\le\cdot\Bigr)\Big|_{s=\ell_{\tau_\varrho}(y)}
\end{equation}
\end{lemma}

\begin{proofsect}{Proof}
This follows from the recursive property of the local time underlying the proof of the Markov property in Lemma~\ref{lemma-M}.
\end{proofsect}

With this we are ready to give:

\begin{proofsect}{Proof of Proposition~\ref{prop-2} for $u>2\sqrt t$}
Adjusting constants in the statement we may assume that $u>2\sqrt{t\vee1}$. Let~$k$ be the largest integer less than~$n$ that obeys
\begin{equation}
k\slb \le 1+u-\sqrt{t} -\frac1{8\slb}\log(t\vee1).
\end{equation}
This is designed to imply
\begin{equation}
\sqrt{t}+a_n(t)-u \le 1+a_{n-k}(1).
\end{equation}
As $\log(t\vee1)\le2\log(1+\sqrt t)\le 2\sqrt t$ while $8\slb\ge6$, we readily check that~$k$ is non-negative and, in fact, larger than a positive constant times~$u$. It thus suffices to show that the probability of interest is exponentially small in~$k$.

Recall that~$\tau_{\BbbL_k}$ is the first hitting time of level~$\BbbL_k$ and that~$\theta_t$ denotes the shift by~$t$ on the sample-path space of the random walk. Then $\tau_{\BbbL_k}':=\tau_{\BbbL_k}\circ\theta_{\tau_{\BbbL_n}}$ is the first time the random walk visits level~$k$ after hitting level~$n$ of the tree. Consider the events
\begin{equation}
\AA_k:=\Bigl\{\sum_{x\in\BbbL_k}1_{\{\ell_{\tau_\varrho}(x)\ge1\}}<k\Bigr\}
\end{equation}
and
\begin{equation}
\BB_k:=
\Bigl\{\sum_{x\in\BbbL_k}1_{\{\ell_{\tau_\varrho}(x)\ge1\}}\ge k\Bigr\}
\cap\Bigl\{\max_{x\in\BbbL_n}\sqrt{\ell_{\tau_\varrho}(x)}- \sqrt t - a_n(t)<-u\Bigr\}.
\end{equation}
The additive structure of the local time yields
\begin{multline}
\quad
\Bigl\{\max_{x\in\BbbL_n}\sqrt{L_t(x)}- \sqrt{t} - a_n(t)<-u\Bigr\}\cap
\bigl\{\tau_{\BbbL_k}'<\wt\tau_\varrho(t)\bigr\}
\\
\subseteq \theta_{\tau_{\BbbL_k}'}^{-1}(\AA_k\cup\BB_k)\cap \bigl\{\tau_{\BbbL_k}'<\wt\tau_\varrho(t)\bigr\}.
\end{multline}
Using that
\begin{equation}
\label{E:2.50}
P^\varrho\Bigl(\HH_{n,t}\,\triangle\,\bigl\{\tau_{\BbbL_k}'<\wt\tau_\varrho(t)\bigr\}\Bigr)=0,
\end{equation}
the strong Markov property at the stopping time~$\tau_{\BbbL_k}'$ gives, for any~$z\in\BbbL_k$,
\begin{equation}
P^\varrho\Bigl(\,\max_{x\in\BbbL_n}\sqrt{L_t(x)}- \sqrt{t} - a_n(t)<-u\,\Big|\, \HH_{n,t}\Bigr)
\le P^z(\AA_k\cup\BB_k),
\end{equation}
where we also used that, by the symmetries of the tree, the choice of~$z$ is immaterial.
The probability $P^z(\AA_k)$ is exponentially small in~$k$ by Lemma~\ref{lemma-1.5}. Writing~$q$ for the double infimum in \eqref{E:1.2}, Lemma~\ref{lemma-2.12} in turn shows
\begin{equation}
P^z(\BB_k)\le \bigl(1-q\bigr)^k.
\end{equation}
It follows that $P^z(\AA_k\cap\BB_k)$ decays exponentially in~$k$ which by the linear relation between~$k$ and~$u$ implies the claim.
\end{proofsect}

\section{Weak convergence of maximal local time}
\label{sec-3}\noindent
In this section we establish the existence of a weak limit of the maximal local time on the leaves for the random walk started at the root and run until the  total  time  spent  there has reached a given value.  In particular, we give proofs of Theorems~\ref{thm-2}, \ref{thm-1.3} and~\ref{thm-3} as well as Corollary~\ref{thm-2b}. 

\subsection{Two main theorems from uniform convergence}
 Let us start with Theorems~\ref{thm-2} and~\ref{thm-1.3}.  As noted earlier, with only a modest amount of additional care  we can show that the weak convergence in \eqref{E:2.3} takes place   uniformly in~$t$ such that~$t=o(n^2)$,  provided we introduce a suitable $t$-dependent part into the centering sequence.
In order to state this version of Theorem~\ref{thm-2}, let 
\begin{equation}
\wt W_n(t):=(t\vee1)^{-1/4}\,\texte^{-2\slb\,\sqrt t}\,\wt Z_n(t),
\end{equation}
 where 
\begin{multline}
\label{E:1.9b}
\wt Z_n(t) := 
b^{-2n} \sum_{x \in \BbbL_n} \biggl(n\slb - (\sqrt{L_t(x)}-\sqrt t)-\frac1{8\slb}\log\frac{L_t(x)\vee1}{t\vee1}\biggr)^+
\\
\times
(L_t(x)\vee1)^{1/4}\,
\rme^{2\slb\, \sqrt{L_t(x)}}.
\end{multline}
 is a slight extension of~$Z_n(t)$ from \eqref{E:1.9a} by terms that cannot be ignored when~$t$ is allowed to vary with~$n$. The uniform control then comes in:

\begin{theorem}
\label{thm-5} 
 Let $a_n(t)$ be as in \eqref{E:1.10a}.  Abbreviating
\begin{equation}
\label{E:3.3w}
F_{n,t}(u):= P^\varrho \Bigl(\,\max_{x\in\BbbL_n} \sqrt{L_t(x)} - \sqrt{t} - a_n(t\vee1) \le u \Bigr),
\end{equation}
there exists~$C_\star\in(0,\infty)$ such that for any positive sequence $\{t_n\}_{n\ge1}$ with~$t_n/n^2\to0$,
\begin{equation}
\label{E:3.3i}
\lim_{k\to\infty}\limsup_{n\to\infty}\sup_{0<t\le t_n}\,\biggl|\,F_{n,t}(u)-
E^\varrho \Bigl(\,\rme^{-C_\star \wt W_k(t) \rme^{-2u\slb}}\,\Bigr)\biggr|=0
\end{equation}
holds for all~$u\in\R$.
\end{theorem}

Before we move to the proof (which will come in Section~\ref{sec-3.2}), let us check how this implies the main convergence results. First we note two consequences of \eqref{E:3.3i}:

\begin{corollary}
\label{cor-3.2}
For each~$t>0$ and each~$u\in\R$, we have
\begin{equation}
\label{E:3.5w}
\lim_{n\to\infty}F_{n,t}(u)=\lim_{k\to\infty}E^\varrho \Bigl(\,\rme^{-C_\star \wt W_k(t) \rme^{-2u\slb}}\,\Bigr)
\end{equation}
where both limits exist. The convergence is uniform on bounded intervals of~$t>0$. Moreover, for each~$t>0$ there is a random variable~$\wt W(t)$ such that
\begin{equation}
\label{E:3.6w}
\wt W_k(t)\,\underset{k\to\infty}\Lawarrow\,\wt W(t).
\end{equation}
\end{corollary}

\begin{proofsect}{Proof from  Theorem~\ref{thm-5}}
The existence of the limits follows from \eqref{E:3.3i}  and the fact $F_{n,t}$ is bounded and independent of~$k$ while the expectation on the right of \eqref{E:3.5w} is bounded and independent of~$n$.  Viewing \eqref{E:3.5w} as the limit of Laplace transforms of $\{\wt W_k(t)\}_{k\ge1}$, the convergence \eqref{E:3.6w} follows from the Curtiss Theorem and the fact, by the tightness proved in Theorem~\ref{thm-4}, the limit quantity in \eqref{E:3.5w} tends to one as~$u\to\infty$. 
\end{proofsect}

The convergence \eqref{E:3.6w} implies weak convergence of random variables~$\{\wt Z_k(t)\}_{k\ge1}$. In order to conclude  Theorem~\ref{thm-2} from this, we need to address the discrepancy between $Z_k(t)$ and~$\wt Z_k(t)$. This is the content of:

\begin{lemma}
\label{lemma-3.1a}
Let~$t>0$ and suppose that the family $\{\wt Z_n(t)\}_{n\ge1}$ is tight under~$P^\varrho$. Then
\begin{equation}
\label{E:3.7}
Z_n(t)-\wt Z_n(t)\,\underset{n\to\infty}{\overset{P^\varrho}\longrightarrow}\,0.
\end{equation}
\end{lemma}

\begin{proofsect}{Proof}
The proof is based on showing that $\wt Z_n(t)$ receives asymptotically vanishing contribution from $x\in\BbbL_n$ for which $n\slb -\sqrt{L_t(x)}<n^\delta$ for some $\delta\in(0,1/4)$. First, Lemma~\ref{lemma-3} relates $\sqrt{2L_t(x)}$ to the value~$B_n$ of standard Brownian motion. Straightforward estimates then show
\begin{multline}
\label{E:3.5o}
E^\varrho\Bigl(\texte^{2\slb\sqrt{L_t(x)}}1_{\{|\sqrt{L_t(x)}-n\slb|<n^\delta\}}\Bigr)
\\
\le\sqrt{\frac t{n\slb-n^\delta}}
\,\,E^{\sqrt{2t}}\Bigl(\texte^{\sqrt{2\log b}\,B_n}1_{\{|B_n-n\sqrt{2\log b}|<\sqrt2\,n^\delta\}}\Bigr)\,.
\end{multline}
A change of variables equates the expectation on the right with 
\begin{equation}
b^n\texte^{2\sqrt{t\log b}}P^0\bigl(|B_n+\sqrt{2t}|<\sqrt2\,n^\delta\bigr).
\end{equation}
As~$\delta<1/2$, this probability is of order $n^{\delta-1/2}$. The expectation on the left of \eqref{E:3.5o} is thus at most a constant times $b^n n^{\delta-1}$ and, since $|\BbbL_n|=b^n$, a routine first-moment estimate shows that the laws of the random variables 
\begin{equation}
\biggl\{\,b^{-2n}n^{1-\delta}\sum_{x\in\BbbL_n}1_{\{|\sqrt{L_t(x)}-n\slb|<n^\delta\}}\texte^{2\slb\sqrt{L_t(x)}}\,\biggr\}_{n\ge1}
\end{equation}
are tight under~$P^\varrho$. 

For $x\in\BbbL_n$ with $|\sqrt{L_t(x)}-n\slb|<n^\delta$, the prefactor of the exponential in \eqref{E:1.9b} is at most $2n^{1/2+\delta}$. As $n^{1/2+\delta}\ll n^{1-\delta}$ thanks to~$\delta<1/4$, the contribution of these points to~$\wt Z_n(t)$ vanishes in the limit as~$n\to\infty$. For the remaining pairs we have $n\slb-\sqrt{L_t(x)}\ge n^\delta$ and, since $\max_{x\in\BbbL_n}\sqrt{L_t(x)}\le n\slb $ with probability tending to one, the contribution of the logarithmic term in \eqref{E:1.9b} to the positive part is negligible as~$n\to\infty$.  As the contribution of~$x\in\BbbL_n$ where $L_t(x)\le1$ is negligible as well, the truncation in~$(L_t(x)\vee1)^{1/4}$ in \eqref{E:1.9b} has vanishing effect as~$n\to\infty$ thus proving \eqref{E:3.7}. 
\end{proofsect}

With this in hand, we are ready to give:

\begin{proofsect}{Proof of  Theorem~\ref{thm-2} from  Theorem~\ref{thm-5}}
Fix~$t>0$. Combining Corollary~\ref{cor-3.2} with Lemma~\ref{lemma-3.1a} we infer the weak convergence \eqref{E:2.2} and conclude that, for each~$u\in\R$,
\begin{equation}
\label{E:3.5ww}
\lim_{n\to\infty}F_{n,t}(u)=\E \Bigl(\,\rme^{-C_\star (t\vee1)^{1/4}\rme^{-2\slb\sqrt t}Z(t) \rme^{-2u\slb}}\Bigr),
\end{equation}
where~$\E$ is the expectation with respect to the law of~$Z(t)$.
Noting that
\begin{equation}
a_n(t\vee1) = n\slb-\frac1{\slb}\log n+ \frac1{8\slb}\log (t\vee1)+o(1),
\end{equation}
this gives \eqref{E:2.3} via shifting~$u$ by the third term on the right. 

It remains to prove \eqref{E:2.2a}. 
For this we first note that $F_{n,t}(u)\ge P^\varrho(\HH_{n,t}^\cc)$ once~$n$ is so large that~$ \sqrt{t} + a_n(t\vee1)+u>0$. Taking~$u\to-\infty$  in \eqref{E:3.5ww}  then shows
\begin{equation}
\label{E:3.20i}
\limsup_{n\to\infty}P^\varrho(\HH_{n,t}^\cc)\le \BbbP\bigl(Z(t)=0\bigr).
\end{equation}
For the complementary inequality  we  note that, for any~$t>0$ and $\epsilon>0$, Theorem~\ref{thm-4} implies  the existence of~$u_0>0$ such that for all~$u>u_0$,
\begin{equation}
P^\varrho \Bigl(\,\max_{x\in\BbbL_n} \sqrt{L_t(x)} - \sqrt{t} - a_n(t\vee1) > -u \Bigr)
\ge (1-\epsilon)P^\varrho(\HH_{n,t}).
\end{equation}
Invoking  \eqref{E:3.5ww}  for the limit~$n\to\infty$ and then  using the Bounded Convergence Theorem to take $u\to \infty$ followed by~$\epsilon\downarrow0$  yields
\begin{equation}
\limsup_{n\to\infty}P^\varrho(\HH_{n,t})\le \BbbP\bigl(Z(t)>0\bigr).
\end{equation}
In conjunction with \eqref{E:3.20i}, this proves \eqref{E:2.2a}.
\end{proofsect}

For Theorem~\ref{thm-1.3}, we need the following fact:

\begin{lemma}
\label{lemma-3.4}
For each~$n\ge1$,
\begin{equation}
\label{E:3.17}
\wt W_n(t)\,\,\underset{t\to\infty}\Lawarrow\,\, W_n,
\end{equation}
where~$W_n$ is defined using the BRW (a.k.a.\ GFF) as in \eqref{E:1.15}.
\end{lemma}

\begin{proofsect}{Proof}
Recall the expression \eqref{E:1.9b} defining~$\wt Z_n(t)$. Using the rewrite
\begin{equation}
\frac{L_t(x)\vee1}t = \biggl[\Bigl(1+\frac{\sqrt{L_t(x)}-\sqrt t}{\sqrt t}\Bigr)\vee\frac1{\sqrt t}\biggr]^2,
\end{equation}
for~$t\ge1$ we can recast $\wt W_n(t)$  as  $b^{-2n}\sum_{x\in\BbbL_n}f_t(\sqrt{L_t(x)}-\sqrt t)$, where
\begin{equation}
f_t(u) := \Biggl(n\slb - u -\frac{\log\bigl[\bigl(1+\tfrac u{\sqrt t}\bigl)\vee \tfrac1{\sqrt t}\bigr]}{4\slb}\Biggr)^+\bigl[\bigl(1+\tfrac u{\sqrt t}\bigl)\vee \tfrac1{\sqrt t}\bigr]^{1/2}\texte^{2\slb\, u}.
\end{equation}
The weak convergence \eqref{E:1.13r} then yields the claim.
\end{proofsect}

\begin{proofsect}{Proof  of  Theorem~\ref{thm-1.3} from  Theorem~\ref{thm-5}}
 Let~$m>0$. Taking $t_n:=m$ in \eqref{E:3.3i}, passing to the limit $n\to\infty$ with the help of local uniformity of the convergence \eqref{E:3.5w}, and then taking~$m\to\infty$ followed by~$k\to\infty$ yields 
\begin{equation}
\label{E:3.16}
\lim_{k\to\infty}\sup_{t>0}\,\biggl|\,\E \Bigl(\,\rme^{-C_\star (t\vee1)^{-1/4}\rme^{-2\slb\sqrt t}Z(t) \rme^{-2u\slb}}\Bigr)-
E^\varrho \Bigl(\,\rme^{-C_\star \wt W_k(t) \rme^{-2u\slb}}\Bigr)\biggr|=0
\end{equation}
for all~$u\in\R$.
 Invoking Lemma~\ref{lemma-3.4} to take $W_k(t)\Lawarrow W_k$ by taking~$t\to\infty$ and  applying that~$W_k\Lawarrow W$ as $k\to\infty$ then  shows
\begin{equation}
\E \Bigl(\,\rme^{-C_\star t^{-1/4}\rme^{-2\slb\sqrt t}Z(t) \rme^{-2u\slb}}\Bigr)\,\,\underset{t\to\infty}\longrightarrow\,\,\E \Bigl(\,\rme^{-C_\star W\rme^{-2u\slb}}\Bigr).
\end{equation}
 As this works for all~$u\in\R$, the  Curtiss Theorem  gives  the convergence \eqref{E:1.18a}. 

It remains to identify the constant~$C_\star$ with the constant~$C_\star'$ in the extremal law \eqref{E:1.12i} for the BRW. Here we will  rely on  the fact proved in Abe~\cite[Corollary 1.3 and Remark~1.4]{A18} that, for any sequence $\{t_n'\}_{n\ge1}$ with $t_n'\ge c_1 n\log n$ and $t_n'/n^2\to0$,
\begin{equation}
\label{E:3.21w}
\lim_{n\to\infty}\,F_{n,t_n'}(u)=\bbE \bigl(\,\rme^{-C_\star' W \rme^{-2u\slb}}\,\bigr).\end{equation}
 Since Abe's regime has a non-trivial overlap with that under which \eqref{E:3.3i} holds,  the weak limits  \eqref{E:3.17} and~$W_n\Lawarrow W$ give
\begin{equation}
\bbE \bigl(\,\rme^{-C_\star' W \rme^{-2u\slb}}\,\bigr)=\bbE \bigl(\,\rme^{-C_\star W \rme^{-2u\slb}}\,\bigr)
\end{equation}
for all~$u\in\R$.
 Hence, $C_\star'W\laweq C_\star W$ and, since $W$ does not vanish a.s.,~$C_\star'=C_\star$. 
\end{proofsect}

\subsection{Uniform convergence from key proposition}
\label{sec-3.2}\noindent
We now move to the proof of Theorem~\ref{thm-5}. The argument follows a strategy that has been used for similar statements for the BRW (A\"idekon~\cite{Aidekon}) as well as the GFF in finite subsets of~$\Z^2$ (Bramson, Ding and Zeitouni~\cite{BDingZ}). A principal input for that strategy is a sharp asymptotics for the right tail of the centered maximum:

\begin{proposition}
\label{prop-3}
There exists $C_\star\in(0,\infty)$ such that the quantity $o(1)=o_{n,t,u}(1)$ defined for integer $n\ge1$ and real $t>0$ and~$u>0$ by
\begin{equation}
\label{E:2.4}
P^\varrho \Bigl(\,\max_{x\in\BbbL_n} \sqrt{L_t(x)} - \sqrt{t} - a_n(t) > u \Bigr)
= C_\star u \rme^{-2u\slb} \bigl(1+o(1)\bigr)
\end{equation}
obeys
\begin{equation}
\lim_{m\to\infty}\sup_{t,u\ge m} \limsup_{n\to\infty}\,\bigl|\,o_{n,t,u}(1)\bigr|=0.
\end{equation}
\end{proposition}


We remark in passing that this strengthens  Abe's lower bound \eqref{E:2.4w} to:

\begin{corollary}
\label{cor-3.6}
There is~$\tilde c' > 0$ such that for all $n\ge 1$, $t\ge 1$ and~$u>0$,
\begin{equation}
\label{E:3.3ww}
P^\varrho\Bigl(\,\max_{x\in\BbbL_n}\sqrt{L_t(x)}- \sqrt t- a_n(t)\ge u\Bigr)\ge \tilde c' (1+u) \texte^{-2u\slb} \,.
\end{equation}
\end{corollary}

Corollary~\ref{cor-3.6} is proved at the very end of Section~\ref{sec-3}.
The proof of Proposition~\ref{prop-3} is long and technical so we will first show how it implies the statement of Theorem~\ref{thm-5}. We start with the following technical fact:

\begin{lemma}
\label{lemma-an}
Fix an integer~$k>1$ and let~$\{t_n\}_{n\ge1}$ be a positive sequence satisfying~$t_n/n^2\to0$. Then, given an integer $n> k$ and reals~$t,s\ge1$, the quantity~$o(1)=o_{n,k,s,t}(1)$ defined by
\begin{equation}
\label{E:2.8}
a_n(t) - a_{n-k}(s) = k\slb +\frac1{8\slb}\log\frac ts + o(1)
\end{equation} 
obeys $o(1) \to 0$ as $n \to \infty$, uniformly in  $t$ and~$s$ satisfying
\begin{equation}
1\le t\le t_n\quad\text{and}\quad 
1\le\sqrt s\le \sqrt t+a_k(t)+\log\log k
\end{equation}
\end{lemma}

\begin{proofsect}{Proof}
A calculation shows
\begin{multline}
\label{E:3.27u}
a_n(t)-a_{n-k}(s) 
\\
= k\slb +\frac1{8\slb}\log\frac ts -\frac3{4\slb}\log\frac n{n-k} -\frac1{4\slb}\log\frac{\sqrt t+n}{\sqrt s+n-k}.
\end{multline}
The claim follows by noting that the last two terms on the right-hand side tend to zero as~$n\to\infty$ uniformly in the above range of~$s$ and~$t$.
\end{proofsect}

We are now ready to give:

\begin{proofsect}{Proof of Theorem~\ref{thm-5} from Proposition~\ref{prop-3}}
The argument is based on the observation that, in order to reach values within~$O(1)$ of~$\sqrt t+a_n(t\vee1)$ at some~$z\in\BbbL_n$, the subtree rooted at its ancestor~$x\in\BbbL_k$ at level~$k$ satisfying $1\ll k\ll n$ must witness an excessively large maximum, and that so even while having a relatively large value of~$L_t(x)$. In order to curb the local time at the ancestral level, for $k \geq 1$ and $t>0$ set
\begin{equation}
\AA_k(t) := \Bigl\{ \max_{x \in \BbbL_k} \sqrt{L_t(x)} \leq \sqrt{t} + a_k(t\vee1) + \log \log k \Bigr\}.
\end{equation}
For $n \geq 1$ and $v \in \bbR$, abbreviate
\begin{equation}
f_{n,t}(v) := P^\varrho \Bigl(\,\max_{x\in\BbbL_n} \sqrt{L_t(x)} - \sqrt{t} - a_n(t\vee1) > v \Bigr) \,.
\end{equation}
Using the Markov property (Lemma~\ref{lemma-M}) to condition on the local time up to level~$k$, we may then  write  $F_{n,t}(u)$ as a quantity of order $P(\AA_k(t)^\rmc)$ plus
\begin{equation}
\label{E:2.7}
E^\varrho \Biggl[\,\prod_{x \in \BbbL_k} \biggl(1 - f_{n-k, L_t(x)}\Bigl(\sqrt{t} + a_n(t\vee1) + u - \sqrt{L_t(x)} - a_{n-k}\bigl(L_t(x)\vee1\bigr)\Bigr) \biggr) \,;\ \AA_k(t) \Biggr]. 
\end{equation}
Lemma~\ref{lemma-M} also shows that $P(\AA_k(t)^\rmc)$ does not depend on~$n$ and, thanks to Lemma~\ref{lemma-2}, tends to zero as $k \to \infty$ uniformly in~$t>0$. We thus need to prove uniform convergence of the expectation \eqref{E:2.7} in the limit as~$n\to\infty$ and~$k\to\infty$.

Assuming that $\AA_k(t)$ occurs, Lemma~\ref{lemma-an} shows that for any~$x\in\BbbL_k$, 
the argument of $f_{n-k,L_t(x)}$ in \eqref{E:2.7} equals
\begin{equation}
\label{E:3.8u}
k\slb - (\sqrt{L_t(x)}-\sqrt t)+\frac1{8\slb}\log\frac {t\vee1}{L_t(x)\vee1}+u+o(1),
\end{equation}
where~$o(1)\to0$ as $n\to\infty$ followed by~$k\to\infty$, uniformly on~$\AA_k(t)$ and in~$t\in(0,t_n]$ The containment in~$\AA_k(t)$ ensures that \eqref{E:3.8u} grows at least as~$c\log k$ as~$k\to\infty$, for some small enough $c>0$, again uniformly on~$\AA_k(t)$. Since the term $u+o(1)$ is order unity, it is thus negligible compared to the rest of the expression. 

Invoking Proposition~\ref{prop-3}, for~$x\in\BbbL_k$ such that $L_t(x)\ge \log\log k$ we thus get
\begin{multline}
\label{E:2.11}
-\log\biggl(1 - f_{n-k, L_t(x)}\Bigl(\sqrt{t} + a_n(t\vee1) + u - \sqrt{L_t(x)} - a_{n-k}\bigl(L_t(x)\vee1\bigr)\Bigr) \biggr)
\\
= \bigl(C_\star + o(1)\bigr) b^{-2k} \biggl(k\slb - (\sqrt{L_t(x)}-\sqrt t)-\frac1{8\slb}\log\frac{L_t(x)\vee1}{t\vee1}\biggr)^+
\\
\times\Bigl(\frac{L_t(x)\vee1}{t\vee1}\Bigr)^{1/4}\rme^{2\slb\, (\sqrt{L_t(x)} - \sqrt{t}-u)},
\end{multline}
where $o(1)$ is a random quantity whose supremum and infimum on~$\AA_k(t)$ tend to zero as $n \to
 \infty$ followed by $k \to \infty$ uniformly in~$t\in(0,t_n]$. For~$x\in\BbbL_k$ where $L_t(x)< \log\log k$ we instead use Lemma~\ref{lemma-2} to get an upper bound where~$C_\star+o(1)$ is replaced by~$2c$, for~$c$ as in~\eqref{E:2.4u}. The resulting quantity is then at most order $b^{-2k}k(\log k)^{O(1)}\le b^{-2k}k^2$, again uniformly in~$t\in(0,t_n]$. 
  
Writing the product in  under expectation  as exponential of sum of the logs  \eqref{E:2.11}  then equates~\eqref{E:2.7} with
\begin{equation}
\texte^{O(b^{-k}k^2)}\bbE \Bigl( \rme^{-(C_\star+o(1)) \wt W_k(t) \rme^{-2\slb \,u}} \,;\, \AA_k(t) \Bigr).
\end{equation}
As~$\wt W_k(t)$ is non-negative and the $o(1)$-term is controlled uniformly on~$\AA_k(t)$, we can push the~$o(1)$ term out of the exponential using the fact that $\rme^{-a(1+o(1))+o(1)} = \rme^{-a} + o(1)$ uniformly in $a \geq 0$. One last application of Lemma~\ref{lemma-2} removes~$\AA_k(t)$ from the expectation and gives
\begin{equation} 
\label{E:2.13}
F_{n,t}(u) = \bbE \bigl(\rme^{-C_\star \wt W_k(t) \rme^{-2\slb \,u}}\bigr) + o(1) \,,
\end{equation}
where $o(1) \to 0$ as $n \to \infty$ followed by $k \to \infty$, uniformly in~$t\in(0,t_n]$.
\end{proofsect}

\begin{remark}
As is readily checked, the estimates in the previous proof are uniform in~$u$ taking values in any  bounded  subset of~$\R$. Thanks to the uniform tightness proved in Theorem~\ref{thm-4}, \eqref{E:3.3i} thus holds even with supremum over~$u\in\R$ inserted between the limit $n\to\infty$ and the supremum over~$t$.  It follows that the convergence \eqref{E:3.6w} takes place in the Kolmogorov metric uniformly in~$t>0$ and, in particular, $W(t)$ is continuously distributed. Similarly, also the centered maximum \eqref{E:3.3w} converges in the Kolmogorov metric on~$\R\cup\{-\infty\}$ uniformly in $t\in(0,t_n]$, for any positive~$\{t_n\}$ with $t_n=o(n^2)$. 

\end{remark}

\begin{remark}
The last term in \eqref{E:3.27u} is no longer negligible when~$t$ grows at least proportionally to~$n^2$. In this regime, the convergence as in \eqref{E:3.21w} still takes place but now with~$C_\star$ multiplied by a term that depends on the asymptotic value of~$t/n^2$; see Abe~\cite[Theorem~1.1]{A18}. While we could include this regime in our computations as well, we refrained from that in order to keep the proofs at manageable length.
\end{remark}

\subsection{Sharp upper tail from technical lemmas}
We will now move to the proof of Proposition~\ref{prop-3}. The argument follows a similar strategy as the proof of Proposition~\ref{prop-1}; namely, we first relate the upper tail probability in~\eqref{E:2.4} to the probability that a Brownian path stays below a random barrier and then use barrier estimates to derive asymptotic for the latter.

To formalize the barrier event, consider a probability space that supports both a Brownian motion
$B = (B_s \colon s \geq 0)$ scaled so that $\Var B_1 = 1/2$ and a collection of independent random  continuous  functions $\{D_k\}_{k=1, \dots, n}$ that are independent of~$B$ and have the law 
\begin{equation}
\label{E:3.23}
\{D_k(s)\}_{s>0} \laweq \Bigl\{\max_{x \in \BbbL'_k} \sqrt{L_{s^2}(x)} - s - a_k(s^2)\Bigr\}_{s>0}.
\end{equation}
Here, as in~\eqref{E:1.8}, $\BbbL'_k$ denotes the set of leaves of $\T_k$ with one child of the root and its sub-tree removed. We shall denote by $\BbbP$ the probability measure on  this  space and specify the initial value of $B$ by (formally) conditioning on~$B_0$.

Given~$u>0$, for $k=0, \dots, n-1$ and $s \geq 0$ set
\begin{equation}
\label{E:3.24}
\Delta_{k}(s) := a_k(s^2) - \frac{k}{n} a_n(t) 
\quad\text{and} \quad 
\wh{B}_s := B_s + \sqrt{t} + u + \frac{s}{n} a_n(t) 
\end{equation} 
and, for $0 \le l \le r \le n$, consider the barrier events
\begin{equation}
\label{E:3.25}
\BB_{l,r}  := \bigcap_{l\le k\le r}\Bigl\{ B_k + D_{n-k}(\wh{B}_k) + \Delta_{n-k}(\wh{B}_k) \leq 0\Bigr\} 
\end{equation}
and
\begin{equation}
\label{E:3.26}
\CC_{l,r} := \bigcap_{s\in[l,r]}\Bigl\{ \big| \wh{B}_s - \sqrt{t} - \slb\ s \big| \leq \tfrac12 \sqrt{t} + \tfrac12 \sqrt{\log b}\, s\Bigr \},
\end{equation}
where $[l,r]$ is an interval in~$\R$.
These are well defined events thanks to the continuity of $s\mapsto D_k(s)$. We note that~$\wh B$ and thus also the events in \twoeqref{E:3.25}{E:3.26} depend on the parameter~$u$, but we will keep that dependence implicit.

The proof of Proposition~\ref{prop-3} will be extracted from three lemmas. The first one expresses the quantity of main interest by way of probabilities of above barrier events:

\begin{lemma}
\label{lemma-2.4}
Recall $F_{n,t}$ from \eqref{E:3.3w}. There exists $C_\sharp \in (0,\infty)$ such that for all~$u\in\R$,
\begin{multline}
\label{E:3.29}
\quad
n^{-1} \rme^{2 \slb\, u}
\frac{\textd F_{n,t}(u)}{\textd u} 
= \bigl(C_\sharp + o(1) \bigr) \BbbP \bigl( \BB_{0,n-1}\cap \CC_{0, n} \,\big|\, B_0 = -u, B_n = 0 \bigr) 
\\
+
O(1) \BbbP \bigl( \BB_{0,n-1}\smallsetminus \CC_{0, n} \,\big|\, B_0 = -u, B_n = 0 \bigr) \,,
\quad
\end{multline}
where the $o(1)$ term tends to $0$ as $n \to \infty$ followed by $t \to \infty$, uniformly in $u > 0$ and the $O(1)$ term is bounded uniformly in these limits. 
\end{lemma}

The other two lemmas supply asymptotic forms for the probabilities on the right hand side of \eqref{E:3.29} in the required limit regime of the parameters:

\begin{lemma}
\label{lemma-2.5}
There exists $C_\diamond \in (0,\infty)$ such that, for all $n\ge1$, $t>0$ and~$u>0$,
\begin{equation}
\label{E:2.20}
\BbbP \bigl( \BB_{0,n-1}\cap \CC_{0, n}  \,\big|\, B_0 = -u, B_n = 0 \bigr) 
= \bigl(C_\diamond+o(1)\bigr) \frac un ,
\end{equation}
where $o(1)\to0$ as $n \to \infty$ followed by $(u,t) \to (\infty,\infty)$.
\end{lemma}

\begin{lemma}
\label{lemma-3.6}
For all $n\ge1$, $t>0$ and $u>0$, 
\begin{equation}
\label{E:3.37a}
\BbbP \bigl( \BB_{0,n-1}\smallsetminus \CC_{0, n}  \,\big|\, B_0 = -u, B_n = 0 \bigr) 
= o(1)\frac{u}{n}  \,,
\end{equation}
where $o(1)\to0$ as $n \to \infty$ followed by $t \to \infty$.
\end{lemma}

With the above three lemmas in hand, we conclude:

\begin{proofsect}{Proof of Proposition~\ref{prop-3} from Lemmas~\ref{lemma-2.4}--\ref{lemma-3.6}}
Combining the above lemmas with the Dominated Convergence Theorem, the probability in~\eqref{E:2.4} equals
\begin{equation}
\bigl(C_\sharp C_\diamond + o(1)\bigr) \int_{u}^\infty s \rme^{-2 \slb\, s} \rmd s \,,
\end{equation}
where $o(1) \to 0$ as $n \to \infty$ followed by $(u,t) \to (\infty, \infty)$.
As the integral evaluates to $(2 \slb)^{-1} u\rme^{-\slb\, u}(1+o(1))$, this gives the claim with $C_\star := (2 \slb)^{-1} C_\sharp C_\diamond$.
\end{proofsect}

It remains to provide the proofs of the above three lemmas. While the latter two require considerable amount of work (which is deferred to the next subsection), the first one is deduced from a familiar calculation:

\begin{proofsect}{Proof of Lemma~\ref{lemma-2.4}}
Fix any $x_n\in\BbbL_n$. Similarly as in the proof of Lemma~\ref{lemma-2.7a}, using that the joint law of the local time at any two vertices of $\T^n$ has no atoms on the positive-real half-line,  the symmetry of the leaves of~$\T_n$ permits us to write
\begin{equation}
\label{E:3.44ui}
\textd F_{n,t}(u) = b^n\, P^\varrho \Bigl(\sqrt{L_t(x_n)} - \sqrt{t} - a_n(t)  = \max_{x \in \BbbL_n} \sqrt{L_t(x)} - \sqrt{t} - a_n(t) \in \rmd u \Bigr)
\end{equation}
 on $\{u>0\}$.  Let, as before, $x_0=\varrho, \dots, x_n$ be the vertices in~$\T_n$ on the unique path from the root to~$x_n$ and write $\BbbL_{n-k}'(x_k)$ for the leaves in $\BbbL_n$ that are descendants of~$x_k$ but not of~$x_{k+1}$. Assume that $\{D_k(\cdot)\}_{k=1,\dots,n}$ are defined on the same probability space as the random walk and that these are independent of each other and of the walk. 

By the Markov property (Lemma~\ref{lemma-M})  and \eqref{E:3.23}  --- and writing, with some abuse, $P^\varrho$ for the joint law of these objects --- the probability  on the right of \eqref{E:3.44ui} equals  \begin{multline}
\label{E:3.30}
P^\varrho \biggl(\,\,\bigcap_{k=0}^{n-1}\Bigl\{ \sqrt{L_t(x_k)} + D_{n-k}\bigl(\sqrt{L_t(x_k)}\bigr) + a_{n-k}\bigl(L_t(x_k)\bigr)
\leq \sqrt{t} + a_n(t) + u\Bigr\}
\\
\cap\Bigl\{\sqrt{L_t(x_n)} - \sqrt{t} - a_n(t)  \in \rmd u \Bigr\}\biggr) \,.
\end{multline}
In light of Lemma~\ref{lemma-3},  this is the product of 
\begin{equation}
\label{E:3.31}
\biggl(\frac{\sqrt{t}}{\sqrt t+ a_n(t)+u}\biggr)^{1/2} 
\BbbP \bigl(B_n - \sqrt{t} - a_n(t)  \in \rmd u \,\big |\, B_0 = \sqrt{t} \bigr) 
\end{equation}
times
\begin{multline}
\label{E:3.32}
\E \biggr( \rme^{-\frac{3}{16} \int_{0}^n B_s^{-2}\textd s} 
\,;\, \bigcap_{0\le k\le n-1} \Bigl\{ B_k + D_{n-k}(B_k) + a_{n-k}(B_k^2) \leq \sqrt{t} + a_n(t) + u\Bigr\} \\
\cap\bigl\{\,\min_{0\le k\le n} B_k > 0\bigr\} \,\bigg|\, B_0 = \sqrt{t},\, B_n = \sqrt{t}+a_n(t) + u \biggr)\,,
\end{multline}
where~$B$ is  $1/\sqrt{2}$-multltiple  of a standard Brownian Motion.

Thanks to the explicit form of the Gaussian law, the Radon-Nikodym derivative of the measure in \eqref{E:3.31} with respect to the Lebesgue measure on $(0,\infty)$ equals
\begin{equation}
\biggl(\frac{\sqrt{t}}{\sqrt t+ a_n(t)+u}\biggr)^{1/2}\frac1{\sqrt{\pi n}}\, \rme^{-\frac{(a_n(t)+u)^2}{n}}
= \frac{1+o(1)}{\sqrt{\pi}\,(\log b)^{1/4}}\, n b^{-n}\, \rme^{-2 \slb\, u} \,,
\end{equation}
where $o(1)\to0$ as $n \to \infty$.
As for the expectation in \eqref{E:3.32}, since $(B_s)_{0\le s\le n}$ under $\BbbP(\cdot\,|\,B_0=\sqrt{t}, B_n=\sqrt{t}+a_n(t)+u)$ is equidistributed to $(\wh{B}_s)_{0\le s\le n}$ under $\BbbP(\cdot\,|\,B_0=-u, B_n=0)$, for $\wh{B}$ as in~\eqref{E:3.24}, the expectation \eqref{E:3.32} can be recast as
\begin{equation}
\label{E:3.34}
\E \Bigl( \rme^{-\frac{3}{16} \int_{0}^n \wh{B}_s^{-2}\textd s} 
\,;\, \BB_{0,n-1}\cap\bigl\{
\min_{0\le k\le n} \wh{B}_k > 0\} \,\Big|\, B_0 = -u,\, B_n = 0 \Bigr)  \,,
\end{equation}
where $\BB_{0,n-1}$ is as in~\eqref{E:3.25}.

Next we observe that, on $\CC_{0,n}$, the integral in the expectation is at most
\begin{equation}
4 \int_{0}^\infty \bigl(\sqrt{t} + \sqrt{\log b}\, s\bigr)^{-2} \rmd s
= \frac{4}{\slb} \,t^{-1/2} \,,
\end{equation}
which tends to $0$ as $t \to \infty$. Since $\min_{0\le k\le n} \wh B_u > 0$ trivially on $\CC_{0,n}$, the part of the expectation~\eqref{E:3.34} restricted to this event equals
\begin{equation}
(1+o(1)) \BbbP \bigl( \BB_{0,n-1}\cap\CC_{0, n} \,\big|\, B_0 = -u, B_n = 0 \bigr) \,,
\end{equation}
with $o(1)\to0$ as $t \to \infty$, uniformly in~$n\ge1$. On the complement of~$\CC_{0,n}$ we instead bound the exponential term in~\eqref{E:3.34} by one and remove the restriction to the event $\{\min_{0\le k\le n} \wh{B}_k > 0\}$ to obtain a quantity of order $\BbbP ( \BB_{0,n-1}\smallsetminus\CC_{0, n} \,|\, B_0 = -u, B_n = 0)$.
Setting $C_\sharp :=\pi^{-1/2}(\log b)^{-1/4}$ we get~\eqref{E:3.29}.
\end{proofsect}

\subsection{Technical lemmas from ballot estimates}
\label{sec-3.4}\noindent
The proofs of Lemmas~\ref{lemma-2.5} and~\ref{lemma-3.6} will require a number of preliminary calculations. We start with two asymptotic bounds for the probability that a Brownian path, modified by auxiliary independent random variables at integer times, stays below a logarithmic curve. We will denote the auxiliary random variables by $\{\wt{D}_k\}_{k=0,\dots,n}$ and refer to them as decorations. Abusing our earlier notation, we  use $\BbbP$ to denote the joint law of these objects under which $B = (B_u \colon  u \geq 0)$ is a Brownian motion scaled so that $\Var(B_1) = 1/2$ and $\{\wt{D}_k\}_{k=0,\dots,n}$ are independent of one another and of~$B$.

\begin{lemma}
\label{lemma-2.6}
Let $C,c \in (0, \infty)$. Suppose that $\BbbP(\wt{D}_k > y) \leq C\rme^{-cy}$ for all $y \geq 0$ and $k \geq 0$. Then for all real $A,u,v\ge0$ and integer $n\ge1$,
\begin{multline}
\label{E:2.23}
\BbbP \biggr(\,\bigcap_{k=0}^{n-1}\Bigl\{B_k + \wt{D}_k + C \bigl(1+\log^+ \bigl(k \wedge (n-k)\bigr)\bigr) \leq 0\Bigr\} \\
\cap \bigcap_{s \in [0,n]}\bigl\{|B_s + u | \leq A + cs  \bigr\} \,\bigg|\, B_0 = -u, B_n = -v \biggr) \geq 4 \frac{uv}{n} \bigl(1+o(1)\bigr) \,,
\end{multline}
where $o(1)\to0$ in the limits as $n \to \infty$ followed by $(A,u,v) \to (\infty, \infty, \infty)$.
\end{lemma}

\begin{lemma}
\label{lemma-2.7}
Let $C,c \in (0, \infty)$. For all $k=0, \dots,n$, suppose that $\wt{D}_k$ takes value~$-C$ with probability~$c$ and equals $-\infty$ otherwise. Then for all real $u,v \ge0$ and integer $n\ge1$,
\begin{multline}
\label{E:2.24}
\BbbP \biggl(\,\bigcap_{k=0}^{n-1}\Bigl\{
B_k + \wt{D}_k - C \bigl(1+ \log^+ \bigl(k \wedge (n-k)\bigr)\bigr) \leq 0 \Bigr\}\,\bigg|\, B_0 = -u, B_n = -v \biggr) \\
\leq 4 \frac{uv}{n} \bigl(1+o(1)\bigr) \,,
\end{multline}
where $o(1)\to0$ as $n \to \infty$ followed by $(u,v) \to (\infty, \infty)$. Moreover, there exists $C' < \infty$, depending on $C$ and $c$ only, such that probability above is always at most
\begin{equation}
\label{e:3.53}
C'\frac{(u^++1)(v^++1)}{n} \,.
\end{equation}
\end{lemma}

 Lemmas~\ref{lemma-2.6}--\ref{lemma-2.7} fall under the umbrella of ballot theorems which are a key technical tool in the subject area of log-correlated fields.  We relegate proofs  of these lemmas  to Section~\ref{sec-5}  that  deals with ballot estimates  needed in this paper  systematically.

In order to compare the barrier probability in~\eqref{E:2.20} with that in the more general setting of Lemmas~\ref{lemma-2.6}--\ref{lemma-2.7}, we will need to control the term $\Delta_{n-k}(\wh{B}_k)$ in the definition of event $\BB_{l,r}$. For this we observe:

\begin{lemma}
\label{lemma-2.8}
There exists $C > 0$ and, for each $t > 1$, also $n(t) < \infty$ such that
\begin{equation}
\CC_{0,n} \subseteq \bigcap_{k=0}^{n}\Bigl\{ |\Delta_{n-k}(\wh{B}_k)| \leq C \bigl(1+ \log^+ \bigl(k \wedge (n-k)\bigr)\bigr) \Bigr\} 
\end{equation}
holds for all $t > 1$ and $n > n(t)$ (and all $u>0$ implicitly contained in the above objects). 
\end{lemma}

\begin{proofsect}{Proof}
The value of~$u$ matters only for the definition of~$\wh B$ and is immaterial in what follows. Let $a_n(\infty)$ be defined as in~\eqref{E:1.10a}, but without the last, $t$-dependent, term.
By~\eqref{E:1.10} with $s=n-k$ and $r=n$, 
\begin{equation}
\label{E:2.26}
\Bigl|\,a_{n-k}(\infty) - \frac{n-k}{n} a_n(\infty)\Bigr|
\leq 1+\log^+ \bigl(k \wedge (n-k) \bigr),
\end{equation}
where we also used that $\slb\ge3/4$ for all~$b\ge2$.
It thus remains to show that, on $\CC_{0,n}$, the right hand side above also bounds 
\begin{equation}
\label{E:3.36}
\bigg| \log \frac{\wh{B}_k + (n-k)}{\wh{B}_k} - \frac{n-k}{n} \log \frac{\sqrt{t}+n}{\sqrt{t}} \bigg| 
\end{equation}
up-to a multiplicative constant that is independent of~$t>1$ as well as~$k$ with $0 \le k \le n$ and $n > n(t)$ for some $n(t) < \infty$. 

Fixing $t > 1$ and~$k$ with~$0\le k\le n$, on $\CC_{0,n}$ we have 
\begin{equation}
\frac12 \bigl(\sqrt{t} + \slb\, k\bigr)\le\wh B_k\le \frac32 \bigl(\sqrt{t} +  \slb\, k\bigr).
\end{equation}
Therefore, there are $c, C \in (0,\infty)$ such that for all~$n$ large enough, the ratio in the first logarithm in~\eqref{E:3.36} is bounded from below and from above by $c$, resp., $C$ times $n/(\sqrt{t}+k)$. In a similar fashion, the ratio in the second logarithm is bounded by absolute multiples of $n/\sqrt{t}$, whenever $n$ is large enough. It follows that if we replace~\eqref{E:3.36} by
\begin{multline}
\label{E:3.37}
\biggl|\, \log \frac{n}{\sqrt{t} + k} - \frac{n-k}{n} \log \frac{n}{\sqrt{t}} \biggr|
\\
\leq \biggl|\, \log \frac{n}{k+1} - \frac{n-k}{n} \log n \biggr|
+
\biggl|\, \log \frac{\sqrt{t}+k}{k+1} - \frac{n-k}{n} \log \sqrt{t} \biggr| \,,
\end{multline}
then we only ``err'' by an additive absolute constant.

We will now estimate the two terms on the right of \eqref{E:3.37} separately. For the first term we again invoke~\eqref{E:1.10} to replace 
$((n-k)/n)\log n$ by $\log (n-k)$ up to an error which is bounded as in~\eqref{E:2.26}. This produces the quantity
\begin{equation}
\biggl|\, \log \frac{(k+1)(n-k)}{n} \biggr| 
\end{equation}
which is again bounded by the right hand side of~\eqref{E:2.26}. The second term on the right of \eqref{E:3.37} is always bounded by $2\log (1+\sqrt{t})$ and therefore also by $2 \log^+ (k \wedge (n-k))$ as soon as $k \in (\sqrt{t}, n-\sqrt{t})$. For~$k$ outside this range, assuming $n\ge\sqrt t\log\sqrt t$, if $k \leq \sqrt{t}$, then this term is at most~$1$ plus
\begin{equation}
\label{E:3.38a}
\biggl|\, \log \frac{\sqrt{t} + k}{(k+1)\sqrt{t}} \biggr|
\le \biggl|\, \log \bigg(\frac{1}{\sqrt{t}} + \frac{1}{k+1} \bigg) \bigg| \le 1 +  \log (k+1)   \leq 2(1+ \log^+ k) \,.
\end{equation}
Finally, for all $k > n-\sqrt{t}$, the second term on the right hand side of~\eqref{E:3.37} tends 
uniformly to $1$ as $n \to \infty$ and thus bounded by an absolute constant as soon as $n$ is large enough. Collecting all the bounds, we get the claim.
\end{proofsect}

The first half of the proof of Lemma~\ref{lemma-2.5} is then the content of:

\begin{lemma}
\label{lemma-3.9}
For~integer~$m,n$ with $0< m<n$ and real~$t,u>0$,
\begin{multline}
\label{E:3.40}
\BbbP \bigl( \BB_{0,n-1}\cap\CC_{0, n}  \,\big|\, B_0 = -u, B_n = 0 \bigr)  \\
=\frac{4u}{n} \bigl(1+o(1)\bigr) \biggl[
\bbE \Bigl(B_{n-m}^- \, ; \,\BB_{n-m,n-1}\cap \CC_{n-m,n}  \,\Big|\, B_0 = -u, B_n = 0 \Bigr) +o(1)\biggr],
\end{multline}
where $o(1)\to0$ in the limit as $n \to \infty$ followed by $(t,u,m) \to (\infty, \infty, \infty)$.
\end{lemma}

\begin{proofsect}{Proof}
Let~$u,t>0$ and~$m$ with $0<m<n$ be fixed. Conditioning on $B_{n-m}$ and using the independence of Brownian increments along with the product structure of the underlying events, the probability in the statement equals
\begin{multline}
\label{E:2.30}
\int 
\BbbP  \bigl(-B_{n-m} \in \rmd s \,\big|\, B_0 = -u, B_n = 0 \bigr) 
\BbbP \Bigl( \BB_{0,n-m-1}\cap  \CC_{0, n-m}  \,\Big|\, B_0 = -u, B_{n-m} = -s \Bigr)  \\
\times \BbbP \Bigl( \BB_{n-m,n-1}\cap \CC_{n-m, n} \,\Big|\, B_{n-m} = -s, B_n = 0 \Bigr).
\end{multline}
Our strategy is to show that the last probability is asymptotic to $\frac{u}n(4+o(1))$ and then wrap the rest of the expression into the expectation modulo an $o(1)$ correction.

Starting with the lower bound, let $\{\wt D_k\}_{k=0}^n$ be independent  with~$\wt D_k$ having the law of Exponential$(\sqrt{\log b})$ plus a constant~$C'$ so large that the upper tail of~$\wt D_k$ exceeds  that of~$D_k(s)$ from~\eqref{E:3.23}, for all~$s>0$. (This is possible thanks to the uniform bound in Lemma~\ref{lemma-2}.) It follows that $(D_{n-k}(\wh{B}_k))_{k=0}^{n-1}$ is stochastically dominated by $(\wt{D}_k)_{k=0}^{n-1}$ conditional on~$\wh B$.
Using Lemma~\ref{lemma-2.8} to control the terms~$\Delta_{n-m}(\wh B_k)$, the last probability in \eqref{E:2.30} is then shown to exceed that in~\eqref{E:2.23} with $A := \sqrt{t}/4$, the parameters $(n-m, s, u)$ in place of $(n,u,v)$ and properly chosen constants~$C$ and~$c$. 

Restricting the resulting integral to $-u \in [m^{1/10}, m]$, Lemma~\ref{lemma-2.6} bounds~\eqref{E:2.30} from below by $\frac{4u}n(1+o(1))$ times
\begin{equation}
\label{E:2.31}
\bbE \Bigl(B_{n-m}^- \, ; \,\BB_{n-m,n-1}\cap \CC_{n-m,n}\cap\bigl\{ -B_{n-m} \in [m^{1/10}, m]\bigr\} \,\Big|\, B_0 = -u, B_n = 0 \Bigr) \,,
\end{equation}
where $o(1) \to 0$ as $n \to \infty$ followed by $(s,t,m) \to (\infty, \infty, \infty)$. In order to remove the event restricting~$-B_{n-m}$ to the interval~$ [m^{1/10}, m]$, straightforward estimates show
\begin{equation}
\label{E:3.64}
\bbE \Bigl(B_{n-m}^- \, ; \bigl\{ -B_{n-m}\not\in [m^{1/10}, m]\bigr\} \,\Big|\, B_0 = -u, B_n = 0 \Bigr)\le C'\bigl(m^{1/5-1/2} + \rme^{-c'm} + u\tfrac mn\bigr),
\end{equation}
for suitably chosen constants~$C',c'>0$. As the right-hand side tends to zero in the stated limits, we get ``$\ge$'' in \eqref{E:3.40}.

For a matching upper bound, we first need to control the lower tail of $D_{n-k}(\wh{B}_k)$. To this end, we first adapt the conclusion of Proposition~\ref{prop-1} to the case when the maximal local time is taken over $\BbbL'_n$ in place of $\BbbL_n$. This is done by noting that for any $s > 1$, $k \ge 1$ and~$C\in\R$, the Markov Property and monotoncity of the local time give
\begin{equation}
\label{E:3.42}
P^\varrho(D_k(s) > -C) \ge P^\varrho \bigl( L_{s^2}(y)\ge s^2\bigr) 
P^\varrho\Bigl(\,\max_{x\in\BbbL_{k-1}}\sqrt{L_{s^2}(x)}\ge s + a_k(s^2) - C \Bigr) \,,
\end{equation}
with $y \in \T_k$ being any child of the root. 
As $L_{s^2}$ is Compound Poisson-Exponential with parameter~$s$, the first probability is readily checked to be uniformly positive in~$s\in[1,m]$, for any~$m\ge1$, and it remains so uniformly on~$[1,\infty)$ because, by \eqref{E:1.13r}, $\sqrt{L_{s^2}(y)}-s$ tends weakly to $\NN(0,1/2)$ as $s\to\infty$. 
Since $a_k(s^2) - a_{k-1}(s^2)$ is bounded by an absolute constant for all $k \ge 1$ and $s> 1$, Proposition~\ref{prop-1} shows that the second term on the right hand side of~\eqref{E:3.42} is uniformly positive, once~$C$ is large enough. 

The latter implies that by properly choosing $C$ and $c$, on $\CC_{0,n}$ and for $t$ sufficiently large, the family $\{\wt{D}_k\}_{k=0}^{n-1}$ from Lemma~\ref{lemma-2.7} is stochastically dominated by $\{D_{n-k}(\wh{B}_k)\}_{k=0}^{n-1}$ conditional on~$\wh B$. Using Lemma~\ref{lemma-2.8} and arguing as before, we now upper bound the 
last probability in~\eqref{E:2.30} by the probability in~\eqref{E:2.24} with $(n-m,s,u)$ in place of $(n,u,v)$ and with properly chosen~$c$ and~$C$. For this probability we then use the asymptotics given by Lemma~\ref{lemma-2.7} on the range of integration~$-s \in [m^{1/10}, m]$ and the upper bound therein for the complementary values of~$s$. This dominates~\eqref{E:2.30} from above by $\frac un(4+o(1))$ multiple of \eqref{E:2.31} plus $\frac unO(1)$-multiple of \eqref{E:3.64}.
Invoking the bound in \eqref{E:3.64} one more time, we then get ``$\le$'' in \eqref{E:3.40}, thus finishing the proof.
\end{proofsect} 

In order to bring the representation~\eqref{E:3.40} to the form stated in Lemma~\ref{lemma-2.5}, we need to show that the expectation admits a limit as $n \to \infty$. 

\begin{lemma}
\label{lemma-2.10}
For all $m \ge 1$, $t > 0$ and $u>0$, 
\begin{multline}
\label{E:3.48}
\bbE \Bigl(B_{n-m}^- \, ; \,\BB_{n-m,n-1}\cap \CC_{n-m,n} \,\Big|\, B_0 = -u, B_n = 0 \Bigr) \\
\underset{n \to \infty} \longrightarrow \bbE \biggl(B_m^- ;\, \bigcap_{k=1}^m\bigl\{ B_k + M_k \leq  \sqrt{\log b}\, k \bigr\} \bigg|\, B_0 = 0 \biggr)  \,,
\end{multline}
where $(M_k)_{k=1}^m$ are independent of~$B$ and one another with law
\begin{equation}
M_k \laweq \max_{x \in \BbbL'_k} h_x 
\end{equation}
with~$h$ denoting the BRW with step distribution~$\NN(0,1/2)$.
\end{lemma}

\begin{proofsect}{Proof}
Fix $u,t>0$ and $m\ge1$. Once~$n$ is large enough, on $\CC_{n-m,n}^\rmc$ there must be $s \in [n-m, n]$ such that $|B_s| > \frac13\sqrt{\log b}\, n$. A standard argument based on the  Reflection Principle  for Brownian  Bridge  then shows
\begin{equation}
\BbbP\bigl(\CC_{n-m,n}^\rmc\,\big|\,B_0=-u, B_n=0\bigr)\le \rme^{-c n^2/m}\,.
\end{equation}
By the Cauchy-Schwarz inequality, the left hand side of~\eqref{E:3.48} with $\CC_{n-m,n}^\rmc$ in place of $\CC_{n-m,n}$ is at most
\begin{equation}
\rme^{-c \frac{n^2}{2m}} \Bigl( \bbE \bigl(B_{n-m}^2 \,\Big|\, B_0 = -u, B_n = 0 \bigr) \Bigr)^{1/2} 
\leq \Bigl( \frac{m(n-m)}{n} + (mu/n)^2 \Bigr)
\rme^{-c \frac{n^2}{2m}}\,.
\end{equation}
As this tends to zero as $n \to \infty$, it suffices to show~\eqref{E:3.48} without the event $\CC_{n-m,n}$ on the left hand side.

Renaming~$B$ under $\bbP(\cdot | B_0 = -u, B_n = 0)$ as $B^{(n)}$
and~$B$ under $\bbP(\cdot | B_0 = 0)$ as $B^{(\infty)}$, the standard way to generate a Brownian bridge from a Brownian path shows
\begin{equation}
\bigl\{B^{(n)}_{n-s}\bigr\}_{s\in[0,n]}\,\laweq \Bigl\{B^{(\infty)}_s - \frac sn (u+B^{(\infty)}_n)\Bigr\}_{s\in[0,n]}
\end{equation}
Using that $B_n^{(\infty)}=o(n)$ as~$n\to\infty$ a.s.\ it follows that all $\{B^{(n)}\}_{n\ge1}$ can be realized on the same probability space so that
\begin{equation}
\label{E:3.71}
\bigl(B^{(n)}_{n-1}, \dots, B^{(n)}_{n-m}\bigr) 
\,\,\underset{n\to\infty}\longrightarrow\,\,
\bigl(B^{(\infty)}_1, \dots, B^{(\infty)}_m\bigr) \quad\text{a.s.}
\end{equation}
We will assume this to be the case for the rest of the proof.

Setting, with some abuse of our earlier notation,
\begin{equation}
f_{k,u}(v) := P^\varrho \Bigl( \max_{x \in \BbbL'_k} \sqrt{L_{u^2}(x)} - u \leq v\Bigr) \,,
\end{equation}
the left hand side of~\eqref{E:3.48}, with $\CC_{n-m,n}$ omitted, is equal to the expectation of
\begin{equation}
\label{E:3.52}
\bigl(B_{n-m}^{(n)}\bigr)^- \prod_{k=1}^m 
f_{k, \wh{B}^{(n)}_{n-k}}\Bigl(-B^{(n)}_{n-k} + \frac{k}{n} a_n(t) \Bigr) 
\,,
\end{equation}
where, in analogy with $\wh{B}_k$ in~\eqref{E:3.24}, we set
\begin{equation}
\wh{B}^{(n)}_k := B^{(n)}_k + \sqrt{t} + s + \frac{k}{n} a_n(t) \,.
\end{equation}
The convergence \eqref{E:3.71} along with $(n-k)a_n(t)/n \to \infty$ as $n \to \infty$ imply that $\wh{B}^{(n)}_{n-k} \to \infty$ a.s.\ as $n \to \infty$ for all $k=1, \dots, m$. Since also $k a_n(t)/n \to \sqrt{\log b}\, k$ as $n \to \infty$, it follows that for any $\epsilon > 0$ and $u > 0$, as soon as $n$ is large enough we have
\begin{multline}
\label{E:3.75}
\inf_{u' \ge u} f_{k, u'}\Bigl(-B^{(\infty)}_k + \sqrt{\log b}\, k- \epsilon \Bigr) 
\le 
f_{k, \wh{B}^{( n)}_{n-k}}\Bigl(-B^{(n)}_{n-k} + \frac{k}{n} a_n(t) \Bigr) \\
\le
\sup_{u' \ge u}
f_{k, u'}\Bigl(-B^{(\infty)}_k + \sqrt{\log b}\, k + \epsilon \Bigr),
\end{multline}
relying also on the monotoncity of $v\mapsto f_{k,u}(v)$. 

Thanks to the weak convergence of $\sqrt{L_{u^2}(\cdot)}-u$ to BRW as~$u\to\infty$ (see \eqref{E:1.13r}) and the continuity of the law of~$M_k$, we have
\begin{equation}
f_{k,u'}(s)\,\underset{u'\to\infty}\longrightarrow\,\BbbP(M_k \le s)
\end{equation}
for each~$s\in\R$. Taking~$u\to\infty$ followed by $\epsilon \downarrow 0$ in \eqref{E:3.75} then shows
\begin{equation}
f_{k, \wh{B}^{(n)}_{n-k}}\Bigl(-B^{(n)}_{n-k} + \frac{k}{n} a_n(t) \Bigr)\,
\underset{n\to\infty}\longrightarrow \,\BbbP \Bigl(M_k \le - B^{(\infty)}_k  + \sqrt{\log b}\, k\,\Big|\, B^{(\infty)} \Bigr)\quad\text{a.s.}
\end{equation}
for each~$k\ge1$. It follows that  under \eqref{E:3.71}  the quantity~\eqref{E:3.52} converges to
\begin{equation}
\label{E:3.56}
\bigl(B^{(\infty)}_m\bigr)^- \prod_{k=1}^m 
\BbbP \Bigl(M_k \le - B^{(\infty)}_k +  \sqrt{\log b}\, k\Bigr)
\end{equation}
almost surely as~$n\to\infty$. In order to turn this into convergence in the mean, we note that the second moment of~\eqref{E:3.52} is at most
\begin{equation}
\bbE \bigl(B_{n-m}^{(n)}\bigr)^2
= \E \Bigl( B_{n-m}^2  \,\Big |\, B_0 = -s,\, B_n = 0 \Bigr)
= m(n-m)/n + (sm/n)^2 \,,
\end{equation}
which is bounded uniformly in $n$. Hence the sequence~\eqref{E:3.52} is uniformly integrable and thus~\eqref{E:3.52} also converges to~\eqref{E:3.56} under expectation. To complete the proof, we identify the expectation of~\eqref{E:3.56} with the right hand side of~\eqref{E:3.48}.
\end{proofsect}

Before we can draw the desired conclusions from the previous lemmas, we need to ensure that the limiting expectation in \eqref{E:3.48} remains positive and finite uniformly in~$m\ge1$. This comes in:

\begin{lemma}
\label{lemma-3.18}
There are $C>c>0$ such that for all~$m\ge1$,
\begin{equation}
\label{E:3.80}
c\le \bbE \biggl(B_m^- ;\, \bigcap_{k=1}^m\bigl\{ B_k +M_k\leq  \slb\,k \bigr\} \bigg|\, B_0 = 0 \biggr)\le C.
\end{equation}
\end{lemma}

The proof again boils down to standard calculations based on ballot theorems and so we defer it to Section~\ref{sec-5}. We are now ready to give: 

\begin{proofsect}{Proof of Lemma~\ref{lemma-2.5}}
Combining Lemmas~\ref{lemma-3.9},~\ref{lemma-2.10} and~\ref{lemma-3.18}, we get
\begin{multline}
\label{E:3.81}
 \frac n{4u}\BbbP \bigl( \BB_{0,n-1}\cap \CC_{0, n}  \,\big|\, B_0 = -u, B_n = 0 \bigr)  \\
= \bigl(1+o(1)\bigr)\bbE \biggl(B_m^- ;\; \bigcap_{k=1}^m\bigl\{B_k + M_k \leq \slb\,k\bigr\} \, \Big|\, B_0 = 0 \biggr),
\end{multline}
where $o(1)\to 0$ as $n \to \infty$ followed by $(t,u,m) \to (\infty, \infty, \infty)$. Note that the left-hand side does not depend on~$m$ while the expectation on the right does not depend on~$n$, $u$ or~$t$.  As both sides are also uniformly positive and finite thanks to Lemma~\ref{lemma-3.18}, it follows that the expectation on the right-hand side converges as~$m\to\infty$. Writing~$C_\diamond$ for the  limit  multiplied by~$4$, the claim follows from \eqref{E:3.81}.
\end{proofsect}

\begin{proofsect}{Proof of Lemma~\ref{lemma-3.6}}
Fix $s \in \R$ and Let
\begin{equation}
\tau := \inf \bigl \{ k \in \{0,\dots, n\} \colon  \CC_{k,n} \text{ holds} \bigr \} \wedge n \,.
\end{equation}
The left hand side of~\eqref{E:3.37a} is at most
\begin{multline}
\label{E:3.39}
\sum_{k=1}^{n} 
\BbbP \bigl( \BB_{0,n-1}\cap\{ \tau = k\}  \,\big|\, B_0 = -u, B_n = 0 \bigr) \\
\le 
\sum_{k=1}^{n}
\BbbP \bigl (\CC^\rmc_{k-1, k} \,\big|\, B_k = -u, B_n = 0 \bigr)  
\sup_{u'}
\BbbP \bigl( \BB_{k,n-1}\cap \CC_{k, n} \,\big|\, B_k = -u' ,\, B_n = 0 \bigr)  \,,
\end{multline}
where the supremum is over $u'$ with~$u' < u + \sqrt{t} + \slb\, k$. Here we conditioned on~$B_k$ and used that $B_k\ge-u -\sqrt{t} - \slb\, k$ on $\CC_{k,n}$.

For all $t$ and $n$ large enough, the first term in the last sum is at most
\begin{equation}
\label{E:3.41}
\begin{split}
\BbbP & \bigg( \max_{k-1\le s\le k} |B_s + u| > \frac{\sqrt{t} + \slb\, k}{3} \,\Big|\,  B_0 = -u ,\,  B_n = 0 \bigg)  \\
& \le
\BbbP \bigg( \max_{0\le s\le k} |B_s| > \frac{\sqrt{t} + \slb\, k}{3} \,\bigg|\, B_0 = 0 \bigg)
+
\BbbP \bigg( \frac{k}{n}|B_n| > \frac{\sqrt{t} + \slb\, k}{4} \,\bigg|\, B_0 = 0 \bigg) \\
& \le C \rme^{-c' \frac{(\sqrt{t} + \slb k)^2}{k}}
\le C' \rme^{-c''(\sqrt{t} + k)} \,.
\end{split}
\end{equation} 
Here, in the first inequality we used that $(B_s+u)_{0\le s\le n}$ under $\bbP(\cdot|B_0 = -u,\,  B_n = 0)$ has the same law as $(B_s - \frac{s}{n} (B_n - u))_{0\le s\le n}$ under 
$\bbP(\cdot|B_0 = 0)$ for the first inequality; the Maximal Principle for Brownian Motion then yields for the second inequality.

For the second term in the sum in~\eqref{E:3.39}, whenever $k \leq \sqrt{n}$ we proceed as in the proof of Lemma~\ref{lemma-3.9}: We use Lemma~\ref{lemma-2.8} and Proposition~\ref{prop-1}
to upper bound the last probability in~\eqref{E:3.39} by 
\begin{equation}
\BbbP \biggl(\,\bigcap_{\ell=0}^{n-k-1}\Bigl\{B_\ell + \wt{D}_{k} - C \bigl(1 + \log^+ \bigl((\ell+k) \wedge (n-\ell-k)\bigr)\Bigr\} \,\bigg|\,
B_0 = -s', B_{n-k} = 0 \biggr) \,,
\end{equation}
with $\wt{D}_k$ as in Lemma~\ref{lemma-2.7} and proper choices of $c$ and $C$. 
Since
\begin{equation}
\log^+ \bigl((\ell+k) \wedge (n-\ell-k)\bigr) \le \log^+ k + \log^+ \bigl(\ell \wedge (n-k-\ell)\bigr) \,,
\end{equation}
we may use the upper bound in Lemma~\ref{lemma-2.7} with $(n-k, u'+\log^+ k, \log^+ k)$ in place of $(n, u, v)$, to get 
\begin{equation}
C \frac{(1+u'+\log^+ k)(\log^+ k + 1)}{n-k}
\le
C'
\frac{(u + \sqrt{t} + k)(\log^+ k)}{n} \,,
\end{equation}
for all~$n$ sufficiently large. When $k > \sqrt{n}$ we simply bound the supremum in~\eqref{E:3.39} by~$1$. Together with~\eqref{E:3.41}, this shows that the sum in~\eqref{E:3.39} is at most a constant times $\rme^{-c \sqrt{t}} s n^{-1}$ as long as $n$ is large enough. This can be made equal to the right hand side of~\eqref{E:3.37a} for a properly chosen $o(1)$.
\end{proofsect}

Having proved Proposition~\ref{prop-3}, we can also give:

\begin{proofsect}{Proof of Corollary~\ref{cor-3.6}}
Taking~$m$ and~$n_0$ such that the quantity  $o_{n,t,u}(1)$  in Proposition~\ref{prop-3} satisfies $|o_{n,t,u}(1)|\le 1/2$ for~$t,u\ge m$ and~$n> n_0$ shows  proves the claim in this range of the parameters. To extend the claim to~$t\in[1,m]$, we use the Markov property to generate a situation that $\sqrt{L_t(x)}\in[m,m+1]$ for~$x$ a neighbor of the root. Then we apply the statement in the subtree thereof while noting that, by \eqref{E:2.8}, $a_n(t)-a_{n-1}(s)$ is bounded uniformly in $n\ge 1$ and $1\le t,s\le m+1$. Using that the left-hand side of \eqref{E:3.3ww} is uniformly positive for $u\in[0,m]$ and $n\ge1$ thanks to Theorem~\ref{thm-4} and the fact that $\inf_{t\ge1}\inf_{n\ge1}P^\varrho(\HH_{n,t})>0$ then proves the claim.
\end{proofsect}

\section{Random walk started from a leaf}
\label{sec-4}
\noindent
The last remaining item to attend to is the proof of Theorems~\ref{thm-1} and~\ref{thm-3}. We will deduce these from Theorem~\ref{thm-2} by way of a limit argument whose key ingredient is:

\begin{proposition}
\label{prop-4.1}
Let~$\nu_n$ be the uniform measure on~$\BbbL_n$. Then for all~$n\ge1$ and~$t>0$,
\begin{equation}
\label{E:4.1}
\biggl\Vert\,P^{\varrho}\Bigl(\bigl\{L_t(x)\bigr\}_{x\in\BbbL_n}\in\cdot\,\Big|\,\HH_{n,t}\Bigr)
-P^{\nu_n}\Bigl(\bigl\{\ell_{\tau_\varrho}(x)\bigr\}_{x\in\BbbL_n}\in\cdot\Bigr)\biggr\Vert_{\text{\rm TV}}\le \frac b{b-1}\,t,
\end{equation}
where~$P^{\nu_n}$ is the law of the chain with initial distribution~$\nu_n$.
\end{proposition}

The proof of Proposition~\ref{prop-4.1} is simple to describe but writing it formally requires some  notation.
Recall that~$\wt\tau_\varrho(t)$ denotes the first time $s\mapsto\ell_s(\varrho)$ reaches~$t$ and that~$\tau_{\BbbL_n}$ is the first hitting time of the set~$\BbbL_n$. Let~$N(t)$ be the total number of jumps taken by $\{X_s\colon 0\le s\le\wt\tau_\varrho(t)\}$ from~$\varrho$ and let $\theta_u$ denote the shift on the path-space of the random walk defined so that $(X\circ\theta_u)_s = X_{s+u}$. We now observe:

\begin{lemma}
\label{lemma-4.2}
For each~$n\ge1$ and~$t>0$,
\begin{multline}
\label{E:4.2}
\quad
P^\varrho\Bigl(\bigl\{\{L_t(x)\}_{x\in\BbbL_n}\in\cdot\bigr\}\cap\HH_{n,t}\cap\{N(t)=1\}\Bigr) \\
= 
P^\varrho\Bigl(\theta^{-1}_{\tau_{\BbbL_n}}\bigl(\{\{\ell_{\tau_\varrho}(x)\}_{x\in\BbbL_n}\in\cdot\}\bigr)\cap\HH_{n,t}\cap\{N(t)=1\}\Bigr).
\quad
\end{multline}
\end{lemma}

\begin{proofsect}{Proof}
Consider a path of the random walk started from~$\varrho$ and run until time~$\wt\tau_\varrho(t)$. On $\HH_{n,t}\cap\{N(t)=1\}$, the path must visit~$\BbbL_n$ and terminate at~$\varrho$ and so it can be decomposed into three parts: a path from~$\varrho$ to the first hitting point on~$\BbbL_n$, a path from this point back to~$\varrho$ and a constant path staying at~$\varrho$. It follows that
\begin{equation}
\label{E:4.3}
P^\varrho\Bigl(\tau_{\BbbL_n}+\tau_\varrho\circ\theta_{\tau_{\BbbL_n}}<\wt\tau_\varrho(t)\,\Big|\,\HH_{n,t}\cap\{N(t)=1\}\Bigr)=1.
\end{equation}
Next observe that no local time is accumulated on~$\BbbL_n$ during the first  and third  part of the path. So we have
\begin{equation}
\bigl\{L_t(x)\bigr\}_{x\in\BbbL_n}=\bigl\{\ell_{\tau_\varrho}(x)\circ\theta_{\tau_{\BbbL_n}}\bigr\}_{x\in\BbbL_n}
\text{ on }\bigl\{\tau_{\BbbL_n}+\tau_\varrho\circ\theta_{\tau_{\BbbL_n}}<\wt\tau_\varrho(t)\bigr\}\cap\{N(t)=1\}.
\end{equation}
In conjunction with \eqref{E:4.3}, this now implies the claim.
\end{proofsect}

The identity \eqref{E:4.2} is all that we need to give:

\begin{proofsect}{Proof of Proposition~\ref{prop-4.1}}
We start by noting that, since the events $\HH_{n,t}$ and $\{\tau_{\BbbL_n}<\wt\tau_\varrho(t)\}$ are equal up to a $P^\varrho$-null set, the strong Markov property gives
\begin{equation}
\begin{aligned}
P^\varrho\Bigl(\theta^{-1}_{\tau_{\BbbL_n}}\bigl(\{\{\ell_{\tau_\varrho}(x)\}_{x\in\BbbL_n}&\in\cdot\}\bigr)\cap\HH_{n,t}\Bigr)
\\
&=E^\varrho\Bigl(1_{\{\tau_{\BbbL_n}<\wt\tau_\varrho(t)\}}
P^{X_{\tau_{\BbbL_n}}}\bigl(\{\ell_{\tau_\varrho}(x)\}_{x\in\BbbL_n}\in\cdot\bigr)\Bigr)
\\
&=P^{\nu_n}\bigl(\{\ell_{\tau_\varrho}(x)\}_{x\in\BbbL_n}\in\cdot\bigr)
P^\varrho(\HH_{n,t}),
\end{aligned}
\end{equation}
where we used  that $X_{\tau_{\BbbL_n}}$ is uniform on~$\BbbL_n$ under~$P^\varrho(\cdot| \tau_{\BbbL_n}<\wt\tau_\varrho(t))$ thanks to the symmetries of the tree.  In conjunction with Lemma~\ref{lemma-4.2} this shows
\begin{multline}
P^{\varrho}\Bigl(\bigl\{L_t(x)\bigr\}_{x\in\BbbL_n}\in\cdot\,\Big|\,\HH_{n,t}\Bigr)
-P^{\nu_n}\Bigl(\bigl\{\ell_{\tau_\varrho}(x)\bigr\}_{x\in\BbbL_n}\in\cdot\Bigr)
\\
=P^{\varrho}\Bigl(\bigl\{\bigl\{L_t(x)\}_{x\in\BbbL_n}\in\cdot\bigr\}\cap\{N(t)\ge2\}
\,\Big|\,\HH_{n,t}\Bigr)
\\
-P^{\nu_n}\Bigl(\theta^{-1}_{\tau_{\BbbL_n}}\bigl(\{\{\ell_{\tau_\varrho}(x)\}_{x\in\BbbL_n}\in\cdot\}\bigr)\cap\{N(t)\ge2\}\,\Big|\HH_{n,t}\Bigr).
\end{multline}
It follows that 
the TV-norm in \eqref{E:4.2} is bounded  by $P^\varrho(N(t)\ge 2|\HH_{n,t})$, which is further bounded the ratio of $P^\varrho(N(t)\ge2)$ and $P^\varrho(\HH_{n,t})$. As~$N(t)$ is Poisson with parameter~$t$, the former probability is at most~$t(1-\texte^{-t})$.  The claim now follows from Lemma~\ref{lemma-2.10H} whereby we know that $P^\varrho(\HH_{n,t})\ge\frac b{b-1}t$. 
\end{proofsect}

We are now ready to prove the convergence of the maximal local time for the walk started from a leaf and run until the root is hit:

\begin{proofsect}{Proof of Theorem~\ref{thm-1}}
For each~$n\ge1$, $u\in\R$ and $t>0$ abbreviate (with some abuse of earlier notation)
\begin{equation}
\EE(u):=\biggl\{\max_{x\in\BbbL_n}\sqrt{\ell_{\tau_\varrho}(x)}\le \slb\, n-\frac1{\slb}\log n +u\biggr\}
\end{equation}
and
\begin{equation}
\FF_t(u):=\biggl\{\max_{x\in\BbbL_n} \sqrt{L_t(x)} \le \slb\,n-\frac1{\slb}\log n + u\biggr\}.
\end{equation}
Thanks to the symmetries of the tree, for any~$x_n\in\BbbL_n$, Proposition~\ref{prop-4.1} gives
\begin{equation}
\label{E:4.10}
\Bigl|\,P^{x_n}\bigl(\EE(u)\bigr)-P^{\varrho}\bigl(\FF_t(u)\,\big|\,\HH_{n,t}\bigr)\Bigr|\le\frac b{b-1}t.
\end{equation}
 Theorem~\ref{thm-2}  shows that
\begin{equation}
\lim_{n\to\infty}P^{\varrho}\bigl(\FF_t(u)\,\big|\,\HH_{n,t}\bigr) = \bbE \Bigl(\rme^{-C_\star Z(t) \rme^{-2u\slb}}\,\Big|\, Z(t)>0\Bigr).
\end{equation}
Since $P^{x_n}(\EE(u))$ does not depend on~$t$,  taking  $n\to\infty$ followed by~$t\downarrow0$ in \eqref{E:4.10} proves that these probabilities converge as~$n\to\infty$ and, in fact,
\begin{equation}
\label{E:4.12}
\lim_{n\to\infty}P^{x_n}\bigl(\EE(u)\bigr) = \lim_{t\downarrow0}\bbE \Bigl(\rme^{-C_\star Z(t) \rme^{-2u\slb}}\,\Big|\, Z(t)>0\Bigr),
\end{equation}
where the limit on the right exists as well. 

Let~$\ZZ(t)$ be a random variable with the law of $C_\star Z(t)$ conditioned on~$Z(t)>0$. Viewing the expectation on the right of \eqref{E:4.12} as the Laplace transform of the law of~$\ZZ(t)$, the Curtiss Theorem implies that $\ZZ(t)$ converges, as $t\downarrow0$, weakly to a non-negative (a.s.-finite) random variable~$\ZZ$ for which
\begin{equation}
\lim_{n\to\infty}P^{x_n}\bigl(\EE(u)\bigr) = \bbE \bigl(\rme^{-\ZZ \rme^{-2u\slb}}\bigr)
\end{equation}
holds true for all~$u\in\R$. Theorem~\ref{thm-4} gives
\begin{equation}
\lim_{u\to-\infty}\,\,\lim_{n\to\infty}\,\,P^\varrho(\FF_t(u)|\HH_{n,t})=0
\end{equation}
and so~$\ZZ>0$ a.s.\ as desired.
\end{proofsect}

It remains to characterize the law of~$\ZZ$. For this we first give:

\begin{proofsect}{Proof of Corollary~\ref{thm-2b}}
Let~$\varphi\colon\R^2\to\R$ be bounded and continuous and let~$T$ be a positive random variable. Writing~$P_T$ for the law of~$T$, with the expectations denoted  similarly,  Fubini-Tonelli's Theorem gives
\begin{equation}
E_T\otimes E^\varrho\varphi\bigl(T,Z_n(T)\bigr)=\int\bigl[ E^\varrho\varphi(t,Z_n(t))\bigr]P_T(T\in\textd t).
\end{equation}
 Note that, as  $t\mapsto Z_n(t)$ is continuous, so is $t\mapsto  E^\varrho  \varphi(t,Z_n(t))$.  By the Portmanteau Theorem, the weak convergence \eqref{E:2.2} implies the pointwise convergence $E^\varrho \varphi(t,Z_n(t))\to \E\varphi(t,Z(t))$, where $\E$ denotes expectation with respect to the law of~$Z(t)$ which, we note, depends on~$t$.  The Bounded Convergence Theorem then shows
\begin{equation}
\label{E:4.16}
E_T\otimes E^\varrho\varphi\bigl(T,Z_n(T)\bigr)\,\underset{n\to\infty}\longrightarrow\,\,\int \bigl[\E\varphi(t,Z(t))\bigr]P_T(T\in\textd t)
\end{equation}
where we noted that, being a pointwise limit of continuous functions, $t\mapsto\E\varphi(t,Z(t))$ is Borel measurable. 

Interpreting the integral in \eqref{E:4.16} as a continuous linear functional on~$C([0,\infty]\times[0,\infty])$ endowed with the supremum norm,  the Riesz representation theorem along with elementary tightness arguments yield existence of a Borel probability measure~$\mu$ concentrated on $[0,\infty)\times[0,\infty)$ such that the integral equals $\int\varphi(t,z)\mu(\textd t\textd z)$ and the first martingal of~$\mu$ is~$P_T$. Letting $(T,Z(T))$ be a pair of non-negative random variables with joint law~$\mu$,  the integral equals $E\varphi(T,Z(T))$.

The weak convergence \eqref{E:1.14w} then follows from \eqref{E:4.16}. To prove the disintegration formula \eqref{E:1.13}, note that letting~$\varphi$ increase monotonically to the indicator of $\EE:=(-\infty,r)\times(-\infty,u)$ shows that $t\mapsto\BbbP((t,Z(t))\in\EE)$ is Borel and \eqref{E:1.13} holds for this set. The extension to all Borel sets $\EE\subseteq\R^2$ then follows from Dynkin's $\pi/\lambda$-Theorem.

It remains to prove the cascade relation \eqref{E:1.15w}, let $x_1,\dots,x_b$ denote the neighbors of~$\varrho$ in~$\T_n$. Then the Markov property in Lemma~\ref{lemma-M} casts the probability in \eqref{E:2.3} as
\begin{equation}
E^\varrho\biggl(\,\prod_{i=1}^b P^\varrho\Bigl(\,\max_{x\in\BbbL_{n-1}}\sqrt{L_s(x)}\le n\slb-\frac1{\slb}\log n+u\Bigr)\Big|_{s=\sqrt{L_t(x_i)}}\biggr).
\end{equation}
Fixing~$k\le n$,  the limit formula \eqref{E:3.3i} in Theorem~\ref{thm-5}  permits us to rewrite this into the form
\begin{equation}
\label{E:4.17a}
E^\varrho \Biggl(\, \prod_{i=1}^b E^\varrho\biggl(\,\exp\Bigl\{- C_\star b^{-2}Z_k(t)\texte^{-2u\slb}\Bigr\}\biggr)\bigg|_{t:=\sqrt{L_t(x_i)}}\Biggr)+o(1),
\end{equation}
where $o(1)\to0$ in the limits~$n\to\infty$ followed by~$k\to\infty$.
 But $\sqrt{L_t(x_1)}\dots,\sqrt{L_t(x_b)}$ are i.i.d.\ Compound Poisson-Exponential with parameter~$t$ under~$P^\varrho$ and so, writing $Z^{(1)}_k,\dots,Z^{(b)}_k$ for independent i.i.d.\ copies of $Z_k$ under~$P^\varrho$, \eqref{E:1.14w} implies
\begin{equation}
\Bigl(Z_k^{(1)}\bigl(\sqrt{L_t(x_1)}\bigr),\dots,Z_k^{(b)}\bigl(\sqrt{L_t(x_1)}\bigr)\Bigr)\,\,\underset{k\to\infty}\Lawarrow\,\,\bigl(Z(T_1),\dots,Z(T_b)\bigr)\,,
\end{equation}
where $Z(T_1),\dots,Z(T_b)$ are i.i.d.\ copies of $Z(T)$, for~$T$ that is a Compound Poisson-Exponential with parameter~$t$. Hereby we conclude that, as $n\to\infty$, \eqref{E:4.17} tends to
\begin{equation}
\E^{\otimes k}\biggl(\exp\Bigl\{-C_\star \texte^{-2u\slb} b^{-2}\sum_{i=1}^b Z(T_i)\Bigr\}\biggr).
\end{equation}
But this  is also equal to  the right-hand side of \eqref{E:2.3} and so \eqref{E:1.15w} follows from the fact that the Laplace transform of a non-negative random variable determines its  law. 
\end{proofsect}

We are finally in a position to conclude our proofs by  giving:

\begin{proofsect}{Proof of Theorem~\ref{thm-3}}
Recall the definition of~$Z_n(t)$ from \eqref{E:1.9a} and let~$\ZZ_n$ denote the quantity in \eqref{E:1.9c}. With the help of the symmetries of the tree, Proposition~\ref{prop-4.1} gives
\begin{equation}
\label{E:4.17}
\Bigl|\,P^{x_n}\bigl(\ZZ_n\in\cdot\bigr)-P^{\varrho}\bigl(C_\star Z_n(t)\in\cdot\,\big|\,\HH_{n,t}\bigr)\Bigr|\le\frac b{b-1}t
\end{equation}
By  Theorem~\ref{thm-2},  the law of~$Z_n(t)$ conditioned on~$\HH_{n,t}$ converges to that of $Z(t)$ conditioned on~$Z(t)>0$. As shown in the proof of Theorem~\ref{thm-1}, this law converges to that of~$\ZZ$ as $t\downarrow0$. Taking $n\to\infty$ followed by~$t\downarrow0$ in \eqref{E:4.17} then readily yields parts~(1) and~(3) of the claim.

For part~(2) we will use the cascade relation \eqref{E:1.15w}. Writing~$U$ for a rate-$1$  Exponential  random variable and~$T_1,\dots,T_b$ for i.i.d.\  Compound  Poisson Exponentials with parameter~$t$, for any measurable function $\varphi\colon\R^b\to[0,1]$  the union bound shows 
\begin{equation}
\Bigl|\,E \varphi(T_1,\dots,T_b) -\texte^{-t}\varphi(0)-t\texte^{-t}\sum_{i=1}^b E\varphi(U e_i)\Bigr|\le  b P\bigl(N(tb)\ge2\bigr),
\end{equation}
where $e_1,\dots,e_b$ are the canonical unit coordinate vectors in~$\R^b$ and where $N(t)$ is Poisson with parameter~$t$. Given a Borel set~$\EE\subseteq(0,\infty)$, applying this to
\begin{equation}
\varphi(t_1,\dots,t_b):=\BbbP^{\otimes k}\biggl(\,b^{-2}\sum_{i=1}^b Z^{(i)}(t_1)\in\EE\biggr),
\end{equation}
 where $Z^{(1)}(t_1),\dots,Z^{(b)}(t_b)$ are independent under $\BbbP^{\otimes k}$ and where $\varphi$  is measurable by Corollary~\ref{thm-2b}  and elementary arguments from analysis, 
the cascade relation \eqref{E:1.15w} along with the disintegration formula \eqref{E:1.13} show
\begin{equation}
\label{E:4.22}
\Bigl|\,\BbbP\bigl(Z(t)\in\EE\bigr)-bt\texte^{-t}\BbbP\bigl(b^{-2}Z(U)\in\EE\bigr)\Bigr|\le  b
P\bigl(N(tb)\ge2\bigr).
\end{equation}
The right-hand side is order~$t^2$ as $t\downarrow0$, so dividing by $\BbbP(Z(t)>0)$, which  this very formula shows  equals $b t\texte^{-t}\BbbP(Z(U)>0)+o(t)$, and taking~$t\downarrow0$ then gives part~(2) via the characterization proved in part~(1).
\end{proofsect}

\section{Ballot lemmas}
\label{sec-5}\noindent
The purpose of this section is to collect statements from the area of ``ballot theorems'' that enter  various  derivations in this paper. We start with a lemma that is used in the proof of uniform tightness; specifically, to estimate the right-hand side of \eqref{E:2.36}:

\begin{lemma}
\label{lemma-A.1}
Let~$f,g\colon[0,\infty)\to[0,\infty)$ be continuous increasing and such that
\begin{equation}
\limsup_{t\to\infty}\frac{\log f(t)}{\log\log t}<\infty\,\,\wedge\,\,\liminf_{t\to\infty}\frac{g(t)}t>0.
\end{equation}
Then for~$B$ denoting the standard Brownian motion and $f_n(s):=f(s\wedge(n-s))$,
\begin{equation}
\lim_{r\to\infty}
\inf_{n\ge1} \,n\,P^0\Bigl(B\le -f_n\text{\rm\ on }[1,n-1]\,\,\wedge\,\, r+B\ge-g\text{\rm\ on }[0,n]\,\Big|\, B_n=0\Bigr)>0
\end{equation}
\end{lemma}

\begin{proofsect}{Proof}
By monotonicty it is sufficient to prove the  claim  with $f(t) := C_0 t^{1/4} - 1$ and $g(t) := c_0 t$ for some $C_0 < \infty$ and $c_0 > 0$. Using standard ballot estimates (see, e.g., Proposition~2.1 in~\cite{CHL17Sup}), there exists $c > 0$ such that,
\begin{equation}
P^0\Bigl(B\le -f_n\text{\rm\ on }[1,n-1] \Big| B_n=0\Bigr)
\ge c n^{-1} \,,
\end{equation}
for all $n \geq 1$.
It is therefore sufficient to prove that
\begin{equation}
\label{E:5.4}
\lim_{r \to \infty} \sup_{n \ge 1} n P^0\Bigl(\bigl\{B\le -f_n\text{\rm\ on }[1,n-1]\bigr\}\smallsetminus\bigl\{ r+B  \ge-g\text{\rm\ on }[0,n]\bigr\}\,\Big|\, B_n=0\Bigr) = 0 \,.
\end{equation}
 As decreasing~$f_n$ to zero only increases this probability, we may and will henceforth assume that~$f:=0$. 

Partitioning according to which unit interval $(\ell-1,\ell]$ (for $\ell\ge1$ integer) the inequality  $r+B\ge -g$  is first violated, that probability above is at most the sum of
\begin{equation}
\label{e:5.14a}
P^0 \Bigr(\inf_{s \in [n^{2/3}, n]} B_s \le -r - c_0 n^{2/3} \,\Big|\, B_0=0 \Bigl)
\end{equation}
and
\begin{multline}
\label{e:5.6a}
\sum_{\ell=1}^{n^{2/3}} \int_{w \ge 0}
P^0 \Bigl(\,\inf_{s \in [\ell-1, \ell]} B_s \le  -r-c_0 (\ell-1)
\,\Big|\, B_\ell = -w \Bigr) \\
\times P^{-w}\Bigl(B \le 0 \text{ on } [0,n-\ell] \,\Big|\, B_{n-\ell}=0\Bigr) 
P^0 \bigl(-B_\ell \in \rmd w \,\big|\, B_n =0 \bigr) .
\end{multline}
For the probability in~\eqref{e:5.14a} we can use the Reflection Principle for Brownian Motion to infer the upper bound  by 
\begin{equation}
2\rme^{-4(r+c_0 n^{2/3})^2/(2n)}
= 2 n^{-1} \rme^{-4(r+c_0 n^{2/3})^2/(2n) + \log n}
\leq n^{-1} \rme^{-c r^{1/3}} \,,
\end{equation}
where the last inequality follows by taking the maximum over~$n$ in the exponent. 
 Even after multiplication by~$n$, the probability in \eqref{e:5.14a} thus tends  to $0$ as $r \to \infty$.

Turning to~\eqref{e:5.6a}, we abbreviate $r_\ell := r+c_0(\ell-1)$ and whenever $w \in [0,r_\ell/3]$, use monotonicity to bound the first probability in the integrand by
\begin{equation}
P^0 \Bigl(B_{\ell-1} < -r_\ell/2 \,\Big|\, B_\ell = -r_\ell/3 \Bigr)  
 +
P^0 \Big(\inf_{s \in [\ell-1, \ell]} B_s \le -r_\ell
\Big|\, B_{\ell-1} = -r_\ell/2, B_\ell = -r_\ell/3 \Big)
\end{equation}
 Both  terms are bounded by $C \rme^{-c r_\ell^2}$, thanks to the Gaussian tail formula for the former and the Reflection Principle for Brownian Motion for the latter. 
For the second probability in the integrand we use the Reflection Principle again  to bound it by~$Cw/n$ from above.  Since the measure in~\eqref{e:5.6a} is dominated by  $C \rme^{-w^2/(2\ell)} \textd w$,  putting all these bounds together, the integral in~\eqref{e:5.6a} is at most $n^{-1}$ times 
\begin{equation}
C \int_{w \geq 0} \big(\rme^{-c r_\ell^2} + 1_{\{w > r_\ell/3\}} \big)
	\rme^{-w^2/(2\ell)} w \rmd w
= C' \ell \rme^{-c' r_\ell^2/\ell} = C' \ell \rme^{-c'' (r+\ell)} \,.
\end{equation}
Summing over $\ell$ with $1\le\ell\le n^{2/3}$ gives a quantity that tends to zero as $r \to \infty$. 
\end{proofsect}

Next we address two ``technical'' ballot lemmas from Section~\ref{sec-3.4} that underpin the proof of Proposition~\ref{prop-3}:

\begin{proofsect}{Proofs of Lemmas~\ref{lemma-2.6}--\ref{lemma-2.7}} 
Estimates such as those in Lemmas~\ref{lemma-2.6}--~\ref{lemma-2.7} are quite common in the vast literature on log-correlated fields. While none of the existing statements seem to imply these estimates directly, Propositions~2.1 and~2.2 of~\cite{CHL17} (whose proofs appear in the supplement~\cite{CHL17Sup} to \cite{CHL17}) do come very close to what we need here. We shall therefore only highlight the few changes necessary in their proofs to obtain the versions that we need here.

Let us first recount what Propositions~2.1 and~2.2 of~\cite{CHL17} give us directly.
Let~$B$ be a standard Brownian motion, denote by $\{\sigma_k \colon k \geq 1\}$ the arrival times of a Poisson process with rate $\lambda > 0$ and let $\{Y_k\colon k\ge1\}$ be independent random variables with uniformly exponentially decaying tails; i.e., for some $c,C>0$,
\begin{equation}
\label{E:5.8}
\sup_{k\ge1}\bbP(|Y_k| > y) \leq C \rme^{-c y},
	\quad y \geq 0.
\end{equation}
We assume assume that $B$, $\{\sigma_k \colon k \geq 1\}$ and $\{Y_k\colon k\ge1\}$ are independent of one another. Then setting $t:=n/2$, $(x,y) :=(-u, -v)$, $\lambda_{t,s} := 1+\log^+(2(s \wedge (t-s)))$ and writing $2\lambda$ in place of $\lambda$, for each $M\le0$ the aforementioned propositions give 
\begin{multline}
\label{e:5.3}	
	\BbbP \biggr(\,\bigcap_{k\colon \sigma_k \in [0,n]}
	\Bigl\{B_{\sigma_k} + 
	Y_k
	+ M \bigl(1+\log^+ \bigl(\sigma_k \wedge (n-\sigma_k)\bigr)\bigr) \leq 0\Bigr\}  \,\bigg|\, B_0 = -u, B_n = -v \biggr) \\
	 = 4 \frac{uv}{n} (1+o(1)) \,,
\end{multline}
with the $o(1)$-term tending to $0$ as $n \to \infty$ followed by $(u,v) \to \infty$, and
\begin{multline}
\label{e:5.4}	
\BbbP \biggr(\,\bigcap_{k\colon \sigma_k \in [0,n]}
\Bigl\{B_{\sigma_k} + 
Y_k 
+ M \bigl(1+\log^+ \bigl(\sigma_k \wedge (n-\sigma_k)\bigr)\bigr) \leq 0\Bigr\}  \,\bigg|\, B_0 = -u, B_n = -v \biggr) \\
\leq C' \frac{(u^++1)(v^++1)}{n} \,,
\end{multline}
for all $n\ge1$, $u,v\in\R$, where~$C'$ depends only on~$C$ and~$c$ in \eqref{E:5.8}, $\lambda$ and~$M$.

Only routine modifications to the proofs of these propositions are needed to extend the above statements to the case when $\{\sigma_k - \sigma_{k-1}\colon k \geq 1\}$ are i.i.d.\ Geometric with success probability $p \in (0,1]$. (Note that $\sigma_k-\sigma_{k-1}$ are Exponentials above.) Moreover, no modification what-so-ever is required to allow $M$ to be positive as well (this is because the second part of Proposition~2.2 in~\cite{CHL17Sup} could still be used).

Equipped with these new versions of~\eqref{e:5.3} and~\eqref{e:5.4} the upper bounds in Lemma~\ref{lemma-2.7} follow immediately if we set $p:=c$ as the success probability in the definition of the sequence $\{\sigma_k\colon k\ge1\}$, let $Y_s := -C$ for all $s \geq 0$ and choose $M:=-C$ where~$C$ and~$c$ are as in the statement of Lemma~\ref{lemma-2.7}. Turning to Lemma~\ref{lemma-2.6}, abbreviate
\begin{equation}
\label{E:5.11a}
q(u,v):=\bbP \biggr(\,\bigcap_{k=\ell}^{n-1}\Bigl\{B_k + \wt{D}_k + C \bigl(1+\log^+ \bigl(k \wedge (n-k)\bigr)\bigr) \leq 0\Bigr\} \,\bigg|\, B_\ell = -u, B_n = -v \biggr).
\end{equation}
Choosing $p:=1$, $M:=C$ and $Y_k := \wt{D}_k^+$, from~\eqref{e:5.3} we get
\begin{equation}
q_n(u,v)= \bigl(4+o(1)\bigr)\frac{uv}n\,.
\end{equation}
Note that this is \eqref{E:2.23} albeit with the event
\begin{equation}
\label{E:5.11}
\bigcap_{s\in[0,n]}\bigl\{|B_s + u | \leq A + cs\bigr\}
\end{equation}
removed on the left-hand side. 

In order to prove Lemma~\ref{lemma-2.6}, we thus need to show that inclusion of \eqref{E:5.11} changes the probability in \eqref{E:5.11a} by at most a term $o(1)\frac{uv}n$. This amounts to bounding the probability in \eqref{E:2.23} where \eqref{E:5.11} is swapped for its complement. Partitioning according to which unit interval $(\ell-1,\ell]$ (for $\ell\ge1$ integer) in~\eqref{E:5.11} the inequality is first violated, that probability is at most the sum of
\begin{equation}
\label{E:5.14}
\bbP \Big(\sup_{s \in [n^{2/3}, n]} |B_s + u| > c n^{2/3} \,\Big|\, B_0=-u, B_n = -v \Big)
\end{equation}
and
\begin{equation}
\label{e:5.6}
\sum_{\ell=1}^{n^{2/3}} \int
\phi_\ell(u,u-w) q_n(u-w,v)\mu_{u,v}^{(\ell)}(\textd w),
\end{equation}
where  
\begin{equation}
\phi_\ell(u,t):=\bbP \Big(\sup_{s \in [\ell-1, \ell]} |B_s+u| > A+c(\ell-1)
\Big|\, B_0 = -u, B_\ell = -t \Big)
\end{equation}
and where $\mu_{u,v}^{(\ell)}$ is the measure
\begin{equation}
\mu_{u,v}^{(\ell)}(\EE):=\bbP \bigl( B_\ell+u \in \EE \,\big|\, B_0=-u, B_n = -v \bigr).
\end{equation}
We will now estimate \eqref{E:5.14} and the terms in \eqref{e:5.6} separately.

The Reflection Principle for Brownian motion bounds the probability in \eqref{E:5.14} by
\begin{equation}
2\rme^{-4(u+n^{2/3})(v+n^{2/3})/n} \leq 2\rme^{-n^{1/3}}
\end{equation}
which is $o(1)\frac{uv}n$ in the stated limits. Next we bound the term $\phi_\ell(u,u-w)$ as follows. Abbreviating
\begin{equation}
A_\ell:=A+c(\ell-1)
\end{equation}
for $w\not\in[-A_\ell/3,A_\ell/3]$ we simply put $\phi_\ell(u,u-w)\le1$. For $w \in [0,A_\ell/3]$, we use symmetry and monotonicity to drop the absolute value around $B_s+u$ at the cost of a multiplicative factor of~2. The resulting quantity is then at most
\begin{multline}
\bbP \Bigl(B_{\ell-1} > -u+A_\ell/2 \,\Big|\, B_0 = -u, B_\ell = -u+A_\ell/3 \Bigr)  \\
 +
\bbP \Bigl(\,\sup_{s \in [\ell-1, \ell]} B_s > -u+A_\ell \,\Big|\, B_{\ell-1}=-u+A_\ell/2, B_\ell = -u+A_\ell/3 \Bigr) 
\end{multline}
Both probabilities above are bounded by $C'\rme^{-c A_\ell^2}$, thanks to the Gaussian tail formula for the former and the Reflection Principle for Brownian motion for the latter. A similar reasoning gives exactly the same bound when~$w \in [-A_\ell/3, 0]$ and so
\begin{equation}
\phi_\ell(u,u-w)\le C'\rme^{-c A_\ell^2} + 1_{\{|w| > A_\ell/3\}}
\end{equation}
for all~$w$ under the integral.

Turning to the two other terms in the integral in~\eqref{e:5.6}, concavity of the logarithm function implies that whenever $\ell < n^{2/3}$ and $u,v \leq -1$, we can use~\eqref{e:3.53} in Lemma~\ref{lemma-2.7} (whose proof was already completed above) to bound
\begin{equation}
q_n(u-w,v)\le C''\frac{(u^- + w^- + 1)(v^-+1)}n \leq C''' \frac{uv}n(w^-+1)
\end{equation}
For the measure we get
\begin{equation}
\mu_{u,v}^{(\ell)}(\textd w)\le C'''' \rme^{-c' w^2/\ell}\textd w
\end{equation}
as the conditional law of $B_\ell + u$ is Gaussian with mean at least~$-1$ and variance at most~$\ell$, whenever~$n$ is large enough, $\ell < n^{2/3}$ and $u,v$ are kept fixed. Putting these bounds together, the integral in~\eqref{e:5.6} is at most a constant times $uv/n$ times
\begin{equation}
\int \big(\rme^{-c A_\ell^2} + 1_{\{|w| > A_\ell/3\}} \big)
	\rme^{-c' w^2/\ell} (w^-+1) \rmd w
= O(1) \ell \rme^{-c' A_\ell^2/\ell} =O(1) \ell \rme^{-c'' (A+ \ell)} 
\end{equation}
for some $c''>0$. Summing over $\ell$ with $1\le\ell\le n^{2/3}$ gives a quantity that tends to zero as $A \to \infty$. As argued before, this is what we needed in order to prove \eqref{E:2.23}. 
\end{proofsect}

Finally, we prove the ballot lemma that ensures that the quantity in \eqref{E:3.81} is uniformly positive and bounded and thus enables the limit argument presented thereafter.

\begin{proofsect}{Proof of Lemma~\ref{lemma-3.18}}
We start with the lower bound.
The known results on upper tails of BRW maximum (see, e.g., Proposition 1.3 in~\cite{Aidekon}) show that
\begin{equation}
\inf_{m\ge1}\BbbP\Bigl(\,\bigcap_{k=1}^m\bigl\{M_k\le\slb \,k\bigr\}\Bigr)>0
\end{equation}
Since $B$ is independent of $\{M_k\}_{k\ge1}$ under~$\BbbP$, the expectation in \eqref{E:3.80} is thus at least this infimum times
\begin{equation}
\label{E:3.82}
\bbE \biggl(B_m^- ;\; \bigcap_{k=1}^m\bigl\{B_k\leq 0\bigr\} \, \Big|\, B_0 = 0 \biggr).
\end{equation}
For~$m>1$, the Reflection Principle bounds the probability of the giant intersection conditional on~$B_1$ and $B_m$ from below by
\begin{equation}
1-\exp\Bigl\{-\frac{B_1^-B_m^-}{2(m-1)}\Bigr\}.
\end{equation}
Assuming also containment in $\{-2\le B_1\le -1\}\cap\{B_m^-\le\sqrt m\}$, this is at least $\frac14 B_m^-/m$ once~$m$ is sufficiently large. It follows that the expectation in \eqref{E:3.82} is at least
\begin{equation}
\E\Bigl(\frac1{4m}(B_m^-)^2\,; \{-2\le B_1\le-1\}\cap\{B_m^-\le\sqrt m\}\Bigr)
\end{equation}
Routine estimates based on the independence of Brownian increments now bound this  away from zero uniformly in~$m>1$.

Moving to the upper bound, here we use that
\begin{equation}
Y:=\inf_{k\ge1}\bigl[M_k-\slb\, k+c_1\log k\bigr]
\end{equation}
is a.s.\ finite with a Gaussian lower provided that~$c_1>\frac34(\log b)^{-1/2}$ (see, e.g., Theorem~1.3 in~\cite{ChenHe}). 
The expectation in \eqref{E:3.80} is then at most
\begin{equation}
\label{E:5.19}
\bbE \biggl(B_m^- ;\, \bigcap_{k=1}^m\bigl\{ B_k \leq  -Y+c_1\log k \bigr\} \bigg|\, B_0 = 0 \biggr)
\end{equation}
A standard ballot estimate (see, e.g., Lemma~2.3 in~\cite{CHL17Sup}) bounds the probability of the giant intersection, conditional on~$Y$ and $B_m$ and restricted to $\{B_m \leq 0\}$, by 
\begin{equation}
c_2\frac{Y^- (B^-_m+Y^-+\log m)}{m} \,,
\end{equation}
for some finite constant~$c_2$ depending only on~$c_1$. Using this estimate, the expectation in \eqref{E:5.19} conditioned on $Y$ is shown to be at most
\begin{equation}
 c_2\Bigl[\frac{\bbE B_m^-}m Y^-+  \frac{\bbE B_m^2}m(Y^-)^2 + \E(B_m^-)\frac{\log m}m\,Y^{-}\Bigr]
\le c'\bigl[Y^-+(Y^-)^2\bigr] 
\end{equation}
The Gaussian tails of $Y$ then ensure that the expectation of the last quantity is finite.
\end{proofsect}


\section*{Acknowledgments}
\nopagebreak\nopagebreak\noindent
This project has been supported in part by the NSF award DMS-1954343, ISF grants No.~1382/17  and~2870/21  and BSF award 2018330. \rm

\bibliographystyle{abbrv}

\end{document}